\documentclass[11pt]{amsart}
\usepackage{amscd}
\usepackage{amsmath}
\usepackage{amsxtra}
\usepackage{amsfonts}
\usepackage{amssymb}
\usepackage{mathbbol}
\usepackage{color}
\usepackage{cite}

\newtheorem{theorem}{Theorem}[section]
\newtheorem{corollary}[theorem]{Corollary}
\newtheorem{Lem}[theorem]{Lemma}
\newtheorem{Prop}[theorem]{Proposition}

\newtheorem{conjecture}[theorem]{Conjecture}
\theoremstyle{definition}
\newtheorem{Def}[theorem]{Definition}
\newtheorem{remark}[theorem]{Remark}

\newtheorem{example}[theorem]{Example}
\theoremstyle{remark}

\renewcommand{\theclaim}{\textup{\theclaim}}

\numberwithin{equation}{section}

\def\openone

{\mathchoice

{\hbox{\upshape \small1\kern-3.3pt\normalsize1}}

{\hbox{\upshape \small1\kern-3.3pt\normalsize1}}

{\hbox{\upshape \tiny1\kern-2.3pt\SMALL1}}

{\hbox{\upshape \Tiny1\kern-2pt\tiny1}}}

\makeatletter

\newbox\ipbox

\newcommand{\diracb}[1]{\left\langle #1\mathrel{\mathchoice

{\setbox\ipbox=\hbox{$\displaystyle \left\langle\mathstrut
#1\right.$}

\vrule height\ht\ipbox width0.25pt depth\dp\ipbox}

{\setbox\ipbox=\hbox{$\textstyle \left\langle\mathstrut
#1\right.$}

\vrule height\ht\ipbox width0.25pt depth\dp\ipbox}

{\setbox\ipbox=\hbox{$\scriptstyle \left\langle\mathstrut
#1\right.$}

\vrule height\ht\ipbox width0.25pt depth\dp\ipbox}

{\setbox\ipbox=\hbox{$\scriptscriptstyle \left\langle\mathstrut
#1\right.$}

\vrule height\ht\ipbox width0.25pt depth\dp\ipbox}

}\right. }

\newcommand{\dirack}[1]{\left. \mathrel{\mathchoice

{\setbox\ipbox=\hbox{$\displaystyle \left.\mathstrut
#1\right\rangle$}

\vrule height\ht\ipbox width0.25pt depth\dp\ipbox}

{\setbox\ipbox=\hbox{$\textstyle \left.\mathstrut
#1\right\rangle$}

\vrule height\ht\ipbox width0.25pt depth\dp\ipbox}

{\setbox\ipbox=\hbox{$\scriptstyle \left.\mathstrut
#1\right\rangle$}

\vrule height\ht\ipbox width0.25pt depth\dp\ipbox}

{\setbox\ipbox=\hbox{$\scriptscriptstyle \left.\mathstrut
#1\right\rangle$}

\vrule height\ht\ipbox width0.25pt depth\dp\ipbox}

} #1\right\rangle}

\newcommand{\B}{\mathcal{B}}
\newcommand{\beq}{\begin{equation}}
\newcommand{\eeq}{\end{equation}}

\def\blfootnote{\xdef\@thefnmark{}\@footnotetext}

\input xy
\xyoption{all}
\usepackage{amssymb}

\def\R{\mathbb{R}}
\def\N{\mathbb{N}}

\def\A{\mathcal{A}}
\def\E{\mathcal E}

\def\Q{\mathbb{Q}}

\def\-{^{-1}}
\def\B{\mathcal{B}}
\def\D{\mathcal{D}}

\def\S{\mathcal{S}}

\def\Z{\mathbb{Z}}
\def\A{\mathcal{A}}

\title{Product-form Hadamard triples and its spectral self-similar measures}

\author{Lixiang An}
\address{[Lixiang An]School of Mathematics and Statistics,
$\&$ Hubei Key Laboratory of Mathematical Sciences,
Central China Normal University,
Wuhan 430079,
P.R. China.}

 \email{anlixianghai@163.com}

\author{Chun-Kit Lai}

\address{[Chun-Kit Lai]Department of Mathematics, San Francisco State University,
1600 Holloway Avenue, San Francisco, CA 94132.}

 \email{cklai@sfsu.edu}

\thanks{The research of Lixiang An is supported by NSFC grant 12171181 and 11971194.}

\begin{document}

\begin{abstract}
    In a previous work by {\L}aba and Wang, it was proved that whenever there is a Hadamard triple $(N,\D,{\mathcal L})$, then the associated one-dimensional self-similar measure $\mu_{N,\D}$ generated by maps $N^{-1}(x+d)$ with $d\in\D$,  is a spectral measure. In this paper, we introduce product-form digit sets for finitely many Hadamard triples $(N, \A_k, {\mathcal L}_k)$ by putting each triple into different scales of $N$. Our main result is to prove that the associated self-similar measure $\mu_{N,\D}$ is a spectral measure. This result allows us to show that product-form self-similar tiles are spectral sets as long as the tiles in the group $\Z_N$ obey the Coven-Meyerowitz $(T1)$, $(T2)$ tiling condition. Moreover, we show that all self-similar tiles with $N = p^{\alpha}q$ are spectral sets, answering a question by Fu, He and Lau in 2015. Finally, our results allow us to offer new singular spectral measures not generated by a single Hadamard triple.  Such new examples allow us to classify all spectral self-similar measures generated by four equi-contraction maps, which will appear in a forthcoming paper. 
\end{abstract}

\maketitle
\section{introduction}

 A  Borel probability measure $\mu$ on ${\mathbb R}^d$ is called a {\it spectral measure} if we can find a countable set $\Lambda\subset{\mathbb R}^d$ such that the set of exponential functions $E(\Lambda): = \{e^{2\pi i \lambda \cdot x}:\lambda\in\Lambda\}$ forms an orthonormal basis for $L^2(\mu)$.  If such $\Lambda$ exists, then $\Lambda$ is called a {\it spectrum} for $\mu$. If a spectral measure $\mu$ is the Lebesgue measure on a measurable set $\Omega$, then we say that $\Omega$ is a {\it spectral set}.
 
 \medskip
 
 It is well-known from classical Fourier analysis that the unit cubes $[0,1]^d$ is a spectral set with spectrum ${\mathbb Z}^d$. Due to the rigid orthogonality criterion for the exponentials, spectral measures are in general rare in nature, but when spectral measures or spectral sets exist, there must be some strong geometric criterion attached to the support of the measure. This intriguing study was first  initiated by Fuglede  in his seminal paper \cite{F1974}. He conjectured that a spectral sets can be characterized geometrically by translational tiling.  The conjecture remained open until 2004 and it was disproved in dimension 3 or higher in its full generality \cite{T2004, KM1, KM2}. Recently, Lev and Matolcsi  proved that Fuglede's assertion is true if $\Omega$ is a convex domain \cite{LM2022}. In the course of their proof, they actually proved that all spectral sets  must admit a ``weak tiling" which is a generalization of translational tiling in its measure theoretic form. In some sense, this provided a geometric characterization to spectral sets. 
 
 \medskip

A complete solution to classifying all spectral measures  is far more complicated. In \cite{HLL2013}, it was proved that a spectral measure must be purely absolutely continuous, purely singular or purely atomic. For absolutely continuous measures, it is necessary that the measure must be a Lebesgue measure supported on a measurable set \cite{DL2014}. Therefore, studying the spectrality of absolute continuous measures is reduced to the original Fuglede's conjecture.  Our focus of this paper will therefore be on the singular spectral measures.

\medskip

 First, we start with the purely atomic spectral measures.  There has been intensive study of the related  Fuglede's conjecture formulated in finite groups and it has been shown to hold for many families thereof  \cite{FFLS2019, MK2017, S2019}. In this paper, for a finite set $\D = \{d_1,...,d_m\}$, let us define the equally-weighted Dirac measure as
$$
\delta_{\D}  = \frac1{m}\sum_{i=1}^m\delta_{d_i}.
$$
The following condition is equivalent to the fact that 
$\D$ is a spectral set in the cyclic group $\Z_N: = \Z/N\Z$.  We refer the reader \cite{MK2017, S2019, L2001,LL2022,M2022} and the reference therein for some recent progress in the Fuglede's conjecture specifically to cyclic groups. 

\begin{Def}\label{def_hada}
We say that $(N,\D,{\mathcal L})$ forms an (integral) Hadamard triple if $\D$ and ${\mathcal L} \subset \Z$ and   $\delta_{\frac{\D}{N}}$ is a spectral measure with spectrum ${\mathcal L}$. Equivalently, 
$$
\frac{1}{\sqrt{m}} \left(e^{2\pi i \frac{d\ell}{N}}\right)_{d\in\D,\ell\in{\mathcal L}}
$$
is a unitary matrix. We will also adopt the notion of {\L}aba and Wang to define $(N,\D)$ is a {\it compatible pair} if there exist ${\mathcal L}$ such that $(N,\D,{\mathcal L})$ is a Hadamard triple. 
\end{Def}


 Studying the spectrality of purely atomic measures is not enough to understand all spectral measures. The first singular spectral measures without any atoms were discovered by 
 Jorgensen and Pedersen \cite{JP1998}. They discovered that the middle fourth Cantor measure is a spectral measure. The main principle of generating spectral self-similar measures were immediately formalized and they will generate many spectral self-similar measures.  Recall that  given a finite collection of maps 
$$
\phi_i(x)=\frac1N(x+d_i), 
$$
${\mathcal D} =\{d_i: i = 1,...,m\}$ for $i=1,...,m$, the (equal-weight) {\it self-similar measure} $\mu=\mu_{N,\D}$ is the unique probability measure such that 
$$
\mu (E)=\sum_{i = 1}^{m} \frac1m \mu (\phi_i^{-1}(E)), \ \forall E \ \mbox{Borel}.
$$
The {\it attractor} is the unique  non-empty compact set $K = K(N,\D)$ satisfying the identity $K = \bigcup_{i=1}^m \phi_i(K)$ (see \cite{H1981}).
Because the contraction ratio of all the maps are the same, the measure $\mu$ can be written as an infinite convolution of discrete measures
$$
\mu_{N,\D}= \delta_{\frac{\D}{N}}\ast\delta_{\frac{\D}{N^2}}\ast....
$$

\begin{example}
{\rm For the middle-fourth Cantor measure, we let  $\D = \{0,2\}$ and ${\mathcal L} = \{0,1\}$. Then $\D$ is a spectral set in $\Z_4$ with $(4,\D,{\mathcal L})$ is a Hadamard triple. Consider the first $k$ convolution of the discrete measures defined in $\mu_{4,\D}$, }
$$
\delta_{\frac{\D}{4}}\ast\delta_{\frac{\D}{4^2}}\ast...\ast\delta_{\frac{\D}{4^k}} = \delta_{\frac{\D_k}{4^k}}
$$
{\rm where $\D_k = \D+4\D+...+4^{k-1}\D$. With a direct check, $(4^k, \D_k, {\mathcal L}_k )$ is also a Hadamard triple. To pass to the limit and show the spectrality of $\mu_{4,\D}$, we need some extra technical analysis. A simple short argument can be found in \cite{DLW2016}.} 
\end{example}

\medskip

{\L}aba and Wang \cite{LW2002} proved that if $(N,\D,{\mathcal L})$ forms a Hadamard triple, then the associated self-similar measure $\mu_{N,\D}$ is a spectral measure.  The result  is also true for self-affine measures in higher dimension in which the integer $N$ is replaced by expansive integer matrices \cite{DHL2019}. Apart from self-similar measures, it is also possible to generate new spectral measures by random convolutions of different Hadamard triples. It was first proposed by Strichartz \cite{S2000}, and recent significant progresses have been made by the authors and many other researchers \cite{AFL2019, AH2014, AHH2019, AHL2015, D2012, DHL2013, DHL2014, DHL2019, FHW2018, FW2017, LW2002}. 

\medskip

\medskip

\subsection{Product-form self-similar tiles and its spectrality}  Despite the fact that Hadamard triples is sufficient to generate spectral self-similar measures, not all spectral self-similar measures can be generated by the Hadamard triples defined in Definition \ref{def_hada}. 

\begin{example}\label{example_product}
{\rm As a simple example, let $N = 4$ and $\D = \{0,1,8,9\}$. Then the self-similar measure $\mu (4,\D)$ is the normalized Lebesgue measure supported on $K = [0,1]\cup [2,3]$. It is also well-known that $K$ is a spectral set with a spectrum $\Z+\{0,1/4\}$. However, since $\D$ is not in distinct representative class of $\Z_4$, there is no ${\mathcal L}$ such that $(N,\D,{\mathcal L})$ forms a Hadamard triple in the sense of Definition \ref{def_hada}.} 
\end{example}

\medskip

The above  example was first observed by Dutkay and Jorgensen \cite{DJ2009}. This  is a special case  of self-similar  tiles  (or self-affine tiles in high dimension). We recall that if $N  = \#\D$ and the attractor has non-empty interior, then the  attractor is a tile of $\R$ by  translation \cite{LW1996-2}. In this case, the attractor $K(N,\D)$ is called a {\it self-similar tile}.  Foundational properties of self-affine tiles were laid down by Lagarias and Wang in the 90s \cite{LW1996-2,LW1996}. They are special type of translational tiles generated by IFS. It has found interesting applications in wavelets and projection of certain product self-similar fractals in certain direction results in such self-similar tiles.

\medskip

A much more challenging  question was  to determine for which $\D$ the attractor forms a self-similar tile for a given integer $N$. Such digits are called the {\it tile digit sets} of $N$.  When $N$ is a prime, $\D$ can only be a complete residue class modulo $N$  and the self-similar tile admits a tiling set  by lattices \cite{B1991, K1992}.  However, when $N$ is not a prime, the classification is much more complicated. Right now, complete classification  is  only available when $N$  is a prime power or  $N  =  p^{a}q$ where $p,q$ are distinct primes \cite{LR2008, LLR2013, LLR2017}. They are  variants of the product-forms introduced  by Lagarias  and Wang \cite{LW1996} whose definition was motivated from \cite{O1978} who studied representation of numbers by different digit systems.

\medskip

Let us introduce some notation for our paper. We will write
$$
A\oplus B = \{a+b: a\in A, b\in B\}
$$
where all elements  $a+b$ are distinct so that $\#(A\oplus B) = (\#A)(\#B)$. If not all elements are distinct, we will write $A+B$ instead. If $a$ is a number, then $a+B = \{a\}\oplus B$.  We will write $A\oplus B \equiv \Z_N (\mbox{mod} \ N)$  if $A\oplus B$ is a complete residue class modulo $N$.  The same is also defined for more summands.
\begin{Def} \cite{LW1996, O1978}
Let  $N\ge 2$ be an integer. We  say that $\D$ is a {\it direct product-form digit set} of $N$ if there exists $\ell_1<\ell_2<...<\ell_k$ such that 
$$
\D = {\mathcal E}_0\oplus N^{\ell_1}{\mathcal E}_1\oplus...\oplus N^{\ell_k}{\mathcal E}_k
$$
where ${\mathcal E}_0\oplus...\oplus{\mathcal E_k}\equiv \Z_N$ (\mbox{mod} \ $N$).
\end{Def}

Nonetheless, product-forms happen in a much more general form in the sense that the summands can be dependent on the elements in the previous level. In the following, for simplicity of our discussion in the introduction, we just define the one-stage situation. The higher stage product-form will be discussed in the later sections. 

\begin{Def}\label{definition_product-form}
Let  $N\ge 2$ be an integer. We  say that $\D$ is a {\it one-stage product-form digit set} of $N$ if there exists a positive integer $r$ such that
$$
\D  = \bigcup_{d_0\in\E_0} (d_0+ N^r\E(d_0))
$$
where $\E_0\oplus\E(d_0)\equiv \Z_N$ for all $d_0\in\E_0$. 
\end{Def}

\begin{example}\label{example_weak-product}
As an example,  it is known that $\{0,1,8,25\}$ is a tile digit set for $N = 4$, indeed,
$$
\{0,1,8,25\} = (\{0\}+ 4\{0,2\})\cup \left(\{1\}+4\{0, 6\}\right)
$$
where it is a product form of one stage with ${\mathcal E}_0 = \{0,1\}$. By definition,  ${\mathcal E}(d_0) = \{0,2\}$ and $\{0,6\}$ for $d_0 = 0, 1$ respectively. Notice that $\{0,1\}\oplus \{0,2\} \equiv\{0,1\}\oplus\{0,6\}\equiv \Z_4$ (mod 4). 
\end{example}
\medskip 
 
The above one-stage product-form is equivalent to the weak product-form \cite{LR2008} . They can cover all tile digit sets if $N = pq$ where $p,q$ are primes.  We note that to cover all tile digit sets for $p^{\alpha}q$, \cite{LLR2017} introduced some higher stage product-form and  higher order  product-form (See Section \ref{section-mod}). 

\medskip

The spectral property of self-similar tiles with product-form tile digit sets was first studied by Fu, He and Lau \cite{FHL2015}. In their paper, it was shown that if the product-form were generated from factors such that  ${\mathcal E}_0\oplus {\mathcal E}_1\oplus...\oplus {\mathcal E}_{k-1} = \{0,1,...,N-1\}$ (it is called a {\it strict product-form}), then the self-similar tiles will be spectral. Strict product-form imposed a rigid structure on all factors ${\mathcal E}_i$ and thus each of the factors must generate a Hadamard triple.  It remains an open question to decide if  self-similar tiles generated by product-forms are spectral measures. This was mentioned explicitly in Question 6.1 of \cite{FHL2015}.

\medskip

\subsection{Overview of the paper.} Motivated from the product-form self-similar tiles, the main purpose of this paper is develop a new class of spectral self-similar measures generated by product-form Hadamard triples that also cover Examples \ref{example_product}, \ref{example_weak-product}. This product-form Hadamard triple combines  distinct Hadamard triples of $N$ into different powers of $N$ while maintaining the contraction ratio being the base scale $1/N$. The precise setup will be presented in Section 2.  Our main technical results (in Theorems \ref{theorem_main1}, \ref{theorem_main2}, \ref{thm_main1}) are to show that the self-similar measures they generated are all spectral measures with a specific structure of its spectrum. These results lead us into several novel applications. 

\medskip
\begin{enumerate}
   \item   We will provide a positive answer to the question raised in \cite{FHL2015} for all tile digit sets for $N = p^{\alpha}q$ and show that all such self-similar tiles are all spectral sets (Theorem \ref{theorem_FHL}). More generally, all these product-form Hadamard triples and spectral self-similar tiles can be produced easily if we are working on Cover-Meyerowitz (CM)-regular cyclic groups (See Section \ref{examples-section}).
   \item We demonstrate the existence of  fractal type product-form digit sets that continue to form  spectral singular self-similar measures despite it does not satisfy the ordinary Hadamard triple definition (for explicit examples, see Example \ref{example24}). In the example, the digit set can be in a  distinct representative (mod $N$), but they are not spectral sets in $\Z_N$. Yet, the corresponding self-similar measure $\mu_{N,\D}$ is still spectral with a spectrum not a subset of $\Z$.
   \item The {\L}aba-Wang conjecture \cite{LW2002} attempted to classify all possible spectral self-similar measures. The right condition for the digit sets $\D$ has never been completely understood. With our product-form Hadamard triples, we are allowed to propose the modified {\L}aba-Wang conjecture (Conjecture \ref{modified_LW}) that might possibly provide the complete classification. We will prove that  the conjecture does hold if $\#\D = 4$ in our next paper.
   \item Our results also shed some light about spectral sets in cyclic groups. The digit sets $\D$ generating the product-form Hadamard triple are not  spectral sets for the base group $\Z_N$, but it is a spectral set in $\Z_{N^k}$ (Proposition \ref{proposition_hada-k}). Hence, product-form Hadamard triples provide us a very flexible way to generate spectral sets in  cyclic groups $\Z_{N^k}$. 
\end{enumerate}

We remark that for non-spectral measures, there has been serious study about which measures admit exponential Riesz bases or Fourier frames. For classical cases, please refer to \cite{KN,DLev2022} and the reference therein. For singular measures,  the assumption of Hadamard triples can also be relaxed to Riesz bases triples or frame triples that guarantee the singular fractal measures admitting exponential Riesz bases or Fourier frame \cite{AFL2019,DEL}.  The product-form Hadamard triple can also be studied in such a relaxed form. We anticipate such a study in the future.

\medskip

\section{Setup and Main Results.} We will rigorously formulate our main results in this section.

\medskip

\subsection{\bf One-stage product-form Hadamard triples.}  We say that $(N,\B_1)$ and $(N,\B_2)$ are {\it equivalent compatible pair} if they share the same spectrum ${\mathcal L}$ so that $(N,\B_1,{\mathcal L})$ and $(N,\B_2,{\mathcal L})$ are Hadamard triples. 

\medskip

\begin{Def}\label{definition_product-form}
Let  $N\ge 2$ be an integer and $\A = \{a_s: s= 0,1...,n-1\}$ be a subset of integers and for each $s$, we have $\B_s$ as another subset of integers. We say that $\D$ is a {\it product-form digit set} generated by Hadamard triples $(N,\A,{\mathcal L}_1)$ and $(N, \B_s,{\mathcal L}_2)$ if there exists $r\ge 0$ such that
\begin{equation}\label{eq_product-form_Hada}
{\mathcal D}= \D_r = \bigcup_{s=0}^{n-1} \left(a_s + N^{r} {\mathcal B}_s\right).    
\end{equation}
and we have the following conditions for ${\mathcal A}$, ${\mathcal B}_s$, ${\mathcal L}_1$ and ${\mathcal L}_2$:
\begin{enumerate}
    \item  $(N, {\mathcal B}_s)$ are equivalent compatible pairs and they share the same spectrum ${\mathcal L}_2$;  
        \item $(N, \A\oplus {\mathcal B}_s, {\mathcal L}_1\oplus{\mathcal L}_2)$ are Hadamard triples for all $s=0,...,n-1$. 
\end{enumerate}
We will call $(N,\D,{\mathcal L}_1\oplus{\mathcal L}_2)$ a {\it product-form Hadamard triple} if there exist $\A,\B_s$ such that $\D$ is written as in (\ref{eq_product-form_Hada}) and (i), (ii) hold.
\end{Def}

 From condition (i) above,  $(N,\B_s)$ are equivalent.  Condition (ii) means that we can combine two different Hadamard triples to form a new larger Hadamard triple.   Producing the product-form in (\ref{eq_product-form_Hada}) means that we put two Hadamard triples in different scales of $N$. 
 When $r\ge 1$, digits in (\ref{eq_product-form_Hada}) are no longer distinct residue classes (mod $N$), so it would not form a Hadamard triple in the ordinary sense in Definition \ref{def_hada}.
Our main theorem is to show however that the associated self-similar measure is still a spectral measure. 
\begin{theorem}\label{theorem_main1}
Let  $(N,\D,{\mathcal L}_1\oplus{\mathcal L}_2)$ be a product-form  Hadamard triple. Then the self-similar measure $\mu_{N,\D}$ is a spectral measure. 
\end{theorem}

We notice that the spectrum $\Lambda$ will not be a subset of integers. However, they will still be inside a finite union of translates of lattices. It turns out that Theorem \ref{theorem_main1} is going to provide the cornerstone of our theory in which all other complicated cases can be reduced to this case. 

\medskip

\subsection{\bf Higher stage product-form.} We now formulate the higher stage product-form Hadamard triples.

\begin{Def}\label{def-prod-form-Had}
We say that $(N,\D, {\mathcal L}_1\oplus....\oplus {\mathcal L}_{k})$ is a ($k$-stage) {\it  product-form Hadamard triple} if there exist positive integers $\ell_1,...,\ell_{k}$ such that $\D = \D^{(k)}$ is generated in the following process:
\begin{equation}\label{eq_product-form_Hada-Section1}
\left\{\begin{array}{lll}
\ \D^{(0)}&=&\E_0\\
\ 
\D^{(1)}& = &\bigcup_{d_0\in\D^{(0)}} \left(d_0+ N^{\ell_1} {\mathcal E}_1(d_0)\right)\\
\
\D^{(2)}& = &\bigcup_{d_0\in\D^{(1)}} \left(d_1+ N^{\ell_1+\ell_2} {\mathcal E}_2(d_1)\right)\\
&\vdots&\\
\ \D^{(k)} &=& \bigcup_{d_{k-1}\in{\mathcal D}^{(k-1)}} \left(d_{k-1}+ N^{\ell_1+...+\ell_{k}} {\mathcal E}_k(d_{k-1})\right),
\end{array}\right.
\end{equation}
and we have the following condition for ${\mathcal E}_{j}({d_{j-1}})$ and ${\mathcal L}_j$:
\begin{enumerate}
    \item $(N,\E_{0},{\mathcal L}_{0}) $ and $ (N,\E_j(d),{\mathcal L}_{j})$ are Hadamard  triples  for all $d\in{\mathcal D}^{(j-1)}, j=1,...,k$.
    \item For all $1\le m\le k$,  $$(N, \E_0\oplus\E_1({d_0})\oplus...\oplus\E_m({d_{m-1}}), {\mathcal L}_0\oplus{\mathcal L}_1\oplus...\oplus{\mathcal L}_{m})$$ and $$(N, \E_{m}(d_{m-1})\oplus\E_{m+1}(d_{m})\oplus...\oplus\E_k({d_{k-1}}), {\mathcal L}_{m}\oplus{{\mathcal L}_{m+1}}\oplus...\oplus{\mathcal L}_{k})$$ are Hadamard triples for all $d_j\in{\mathcal D}^{(j)}$ with $j = 1,...,k-1$.
\end{enumerate}
\end{Def}

With certain amount of careful work, we will show that $k$-stage product-form Hadamard forms can be reduced to a one-stage Hadamard triples with some higher power of $N$. Applying Theorem \ref{theorem_main1}, we obtain that 
\begin{theorem}\label{theorem_main2}
Let $(N,\D, {\mathcal L}_0\oplus....\oplus {\mathcal L}_{k})$ be a k-stage product-form Hadamard triple. Then the self-similar measure $\mu_{N,\D}$ is a spectral measure. 
\end{theorem}
\medskip

These results allow us to give an answer to the spectrality of product-form tile digit sets of $N = p^{\alpha}q$ which was studied previously by Fu, He and Lau \cite{FHL2015}. 

\begin{theorem}\label{theorem_FHL}
Let $N = p^{\alpha}q$ and $\D$ be a tile digit set of $N$. Then the self-similar tile $K(N,\D)$ is a spectral set. 
\end{theorem}

Tile digit sets of $p^{\alpha}q$ were completely classified in \cite{LLR2017}. One of the main difficulty to showing its spectrality was the existence of higher order product-forms. In this paper, we will show that by multiplying an appropriate factor (g.c.d.($\D$) is no longer equal to 1),  these higher order product-forms can be reduced also to a first order product-form of $\alpha$-stage that we just defined. This product-form also generates product-form Hadamard triple, so we can apply Theorem \ref{theorem_main2} and  prove Theorem \ref{theorem_FHL}.

\medskip
To further generate product-form Hadamard triples, we will need the Coven-Meyerowitz tiling theory \cite{CM1999}. We have noticed that $(N,\D,{\mathcal L})$ forms a Hadamard triple if and only if $\D$ (mod $N$) is a spectral set in the finite cyclic group $\Z_N$ and it has a spectrum in ${\mathcal L}$. A spectral set in $\Z_N$ has close relationship to a tile in $\Z_N$ through two algebraic conditions $(T1)$ and $(T2)$ introduced by Coven and Meyerowitz.   In Section 7, we give a brief survey of the subject. In particular, we will see that a cyclic group with Coven-Meyerowitz condition naturally generates product-form Hadamard triples.

\subsection{Four-digit self-similar measures.} It is clear that number of digits in the digit sets must be a composite number in order for product-form Hadamard triple to exist. In  the following, we specifically study the four-digit self-similar measures. Let $0<\rho<1$ and $\D  =  \{0<a<b<c\}$ so that we  have four contraction maps
$$
f_1(x) =  \rho x, \ f_2(x) = \rho (x+a), \ f_3(x) = \rho(x+b), \ f_4(x) = \rho (x+c).
$$
The unique equal-weighted self-similar measure generated by $\rho$ and $\D$ is denoted by  $\mu_{\rho,\D}$. Using the result we obtain in this paper, by writing back to product-form Hadamard triple of stage one with respect to $N$ below,  we obtain the following cases are spectral self-similar measures.

\begin{theorem}\label{theorem_four_digit}
Let $\rho  = \frac{1}{N}$ where $N= 2^{\beta}m$ for a unique integer $\beta\ge1$ and  $m$ is odd and let 
$$
\D  = \{0,a,  2^t\ell, a+2^t\ell'\}
$$
where  $a,\ell,\ell'$ are positive odd integers and $t$ is not divisible by $\beta$. Then $\mu_{\rho,\D}$ is a spectral self-similar  measure. 
\end{theorem}
    
If $t\ge \beta$, then $\D$ would not be a complete residue modulo $N$ and they are spectral due to the fact that they form a product-form Hadamard triple with some ${\mathcal L}_1\oplus{\mathcal L}_2$. We will  show that the converse also holds, meaning that all four-digit self-similar measures must be of the form $\mu_{\rho,\D}$ where $\rho = (2^{\beta}m)^{-1}$ and $\D$ as described in the theorem. As the proof of the converse requires entirely different techniques developed subsequently by different authors and  also  require us to understand a new  case that was never studied previously,  we will put it in a  forthcoming paper \cite{AHL2022}. 

\medskip

The rest of the paper is organized as follows.  In Section 3,  we will lay out some preliminary reductions of the problem and the general strategy to prove Theorem \ref{theorem_main1} for one-stage product-form Hadamard triple. In Section 4 and 5, we will carefully execute the strategy we proposed. In Section 6, we will study the $k$-stage product-form Hadamard triples by showing that it can be reduced to a one-stage Hadamard triple. In Section 7, we will discuss the modulo product-forms of tile digit sets and prove Theorem \ref{theorem_FHL}. In Section 8, we will prove Theorem \ref{theorem_four_digit}, provide some general methods to produce product-form Hadamard triples and then discuss the new examples that are not known previously.

\section{Preliminaries}

In this paper, $e(x) = e^{2\pi i x}$. The Fourier transform of a Borel probability measure is defined to be 
$$
\widehat{\mu}(\xi) = \int e (-\xi x)d\mu (x).
$$
If ${\mathcal A}$ is a finite set of real number, we will write 
$$
M_{\A}(\xi) = \widehat{\delta_{\A}}(\xi) = \frac{1}{\#\A} \sum_{a\in \A} e(-a\xi).
$$
We record the following equivalent conditions, whose proof is now standard, so it will be omitted. 

\begin{Lem}
The following are equivalent.
\begin{enumerate}
    \item $(N,\D,{\mathcal L})$ forms a Hadamard triple. 
    \item $\delta_{\frac{\D}{N}}$ is a spectral measure with spectrum ${\mathcal L}$.
    \item $$
    \sum_{\ell\in{\mathcal L}} |M_{\frac{\D}{N}}(\xi+\ell)|^2 = 1, \ \forall \ \xi\in{\mathbb R}. 
    $$
\end{enumerate}
\end{Lem}

The self-similar measure $\mu = \mu_{N,\D}$ satisfies the following infinite product formula
$$
\widehat{\mu}(\xi) = \prod_{j=1}^{\infty} M_{\D} \left(\frac{\xi}{N^j}\right).
$$
Moreover, we can put it into finite term $\mu = \mu_p\ast \mu_{>p}$ where 
$$
\mu_p = \delta_{\frac{\D}{N}}\ast...\ast \delta_{\frac{\D}{N^p}}
$$
and $\mu_{>p}$ are the rest of the convolutions. The Fourier transform enjoys also the same product formula. 

\subsection{Preliminary reduction.}  There are several preliminary reductions we can do to prove Theorem \ref{theorem_main1}. 

\medskip

\noindent(1). There is no loss of generality to assume $r = 1$ for the one stage product-form. As noticed in the introduction, if $(N,\A,{\mathcal L }_1)$ is a Hadamard triple, then  $\A$  and ${\mathcal L }_1$ must be in distinct residue classes (mod $N$). Therefore, the following are direct summands, 
$${\bf A}={\mathcal A}\oplus N{\mathcal A}+\cdots\oplus N^{r-1}{\mathcal A}.$$ 
$${\bf L}_1={\mathcal L}_1\oplus N{\mathcal L}_1+\cdots\oplus N^{r-1}{\mathcal L}_1,\quad {\bf L}_2={\mathcal L}_2\oplus N{\mathcal L}_2\oplus\cdots\oplus N^{r-1}{\mathcal L}_2.
$$
Suppose that a product-form Hadamard triple $\D_r$ in  (\ref{eq_product-form_Hada}) is given, we define
\begin{eqnarray}\label{D_factorize}
{\bf D}_1&=&\D_r+N\D_r+\cdots+N^{r-1}\D_{r}\nonumber\\
&=&\bigcup_{i_1, \cdots,i_r=0}^{n-1}\left({a_{i_1}}+Na_{i_2}+\cdots+N^{r-1}a_{i_{r}}+N^r({\mathcal B}_{i_1}\oplus N{\mathcal B}_{i_2}\oplus \cdots\oplus N^{r-1}{\mathcal B}_{i_{r-1}})\right).
\end{eqnarray}
The following lemma is known, whose proof can be found in \cite[Lemma 2.5]{S2000}. A more general version of the lemma will also be proved in Proposition \ref{proposition_hada-k}.

\medskip

\begin{Lem}\label{reduction lemma}
Suppose that $(N,{\mathcal A}, {\mathcal L}_1)$ and  $(N,{\mathcal B}_i, {\mathcal L}_2)$ are Hadamard triples. Then $(N^r,{\bf A}, {\bf L}_1)$ and  $(N^r,\bigoplus_{j=1}^rN^{j-1}{\mathcal B}_{i_j}, {\bf L}_2)$ also form integral Hadamard triples for any $0\le i_1, \cdots, i_r\le n-1$.
\end{Lem}

\medskip

Consequently, we factorize the self-similar measure $\mu = \mu_{N,{\mathcal D}_r}$ as
\begin{eqnarray}\label{eq_factorize}
    \mu_{N,{\mathcal D}_r} &=& \delta_{\frac{1}{N^r}{\bf D}_1}\ast\delta_{\frac{1}{N^{2r}}{\bf D}_1}\ast\delta_{\frac{1}{N^{3r}}{\bf D}_1}\ast\cdots\nonumber\\
    &=& \mu_{N^r,{\bf D}_1}.
\end{eqnarray}
Hence, if one can show that $\mu_{N,\D_1}$ is a spectral measure, then we can apply the result to $\mu_{N^r,{\bf D}_1}$ to show $\mu_{N,\D_r}$ is a spectral measure. Hence,  without loss of generality, we only need to consider the case $r=1$. 

\medskip

\noindent (2). We can assume  without loss of generality that $0\in \B_s$ for all $s = 0,1,...,n-1$. If this is not the case, we take $b_s$ to be the smallest element in $\B_s$, then 
$$
\D_r = \bigcup_{s=0}^{n-1} ((a_s+N^rb_s) + N^r (\B_s-b_s)). 
$$
Letting $\widetilde{\A} = \{a_s+N^rb_s: s = 0,1,...,n-1\}$ and $\widetilde{B}_s = \B_s-b_s$, then $(N,\widetilde{\A}, {\mathcal L}_1)$ and $(N,\widetilde{\B}_s, {\mathcal L}_2)$ are still Hadamard triples. Moreover,  one can show that $\D_r$ is still a one-stage product-form Hadamard triple. 

\medskip

\noindent(3).  We can assume  without loss of generality that g.c.d.$(\D_0) = 1$, where $$\D_0 = \bigcup_{s=0}^{n-1} (a_s+\B_s).$$ If $g = \mbox{g.c.d.}(\D_0) >1$. Then $g$ divides all $a_s+b_s$. In particular, $g$ divides $a_s$ since we assume $0\in\B_s$. Hence, $g$ also divides $b_s$ for all $b_s\in\B_s$. Therefore, 
$$
\D_1 = g\widetilde{\D}_1, \ \widetilde{\D}_1 = \bigcup_{s=0}^{n-1}(\widetilde{a}_s+N \widetilde{\B}_s)
$$
where $\widetilde{\A} = \{\widetilde{a}_s\}$ and $(N, \widetilde{\A}, g{\mathcal L}_1)$ and $(N, \widetilde{B}_s, g{\mathcal L}_2)$, $(N, \widetilde{\A}\oplus\widetilde{B_s}, g{\mathcal L}_1\oplus g{\mathcal L}_2)$ are Hadamard triples. Now, g.c.d.$(\widetilde{\D}_0) = 1$ where $\widetilde{\D}_0 = \bigcup_{s=0}^{n-1}(\widetilde{a}_s+ \widetilde{B}_s$). If one can show that $\mu_{N,\widetilde{\D}_1}$ is a spectral measure with spectrum $\Lambda$, then $\mu_{N,\D_1}$ is also a spectral measure with spectrum $\frac{1}{g}\Lambda$. 

\medskip

Because of the three reductions, we will assume that 
\begin{equation}\label{eqD_1}
\D  = \D_1 = \bigcup_{s=0}^{n-1} \left( a_s + N {\mathcal B}_s\right)    
\end{equation}
where $0\in \B_s$ and g.c.d.$(\D_0)$=1 where $\D_0 = \bigcup_{s=0}^{n-1}(a_s+\B_s)$.

\subsection{Strategy of the proof of Theorem \ref{theorem_main1}.} From the Jorgensen-Pedersen lemma \cite{JP1998}, $\Lambda$ is a spectrum for a probability measure $\mu$ if and only if 
$$
Q(\xi) = \sum_{\lambda\in\Lambda} |\widehat{\mu}(\xi+\lambda)|^2 = 1, \ \forall \xi\in \R.
$$
In this paper, Instead of showing directly $\mu_{N,\D}$ is spectral, we are going to show the following theorem. 
\begin{theorem}\label{thm_main1}
Let $\mu = \mu_{N,\D}$ where $\D$ is a one-stage  product-form Hadamard triple satisfying (\ref{eqD_1}). Then there exists $\Lambda\subset\Z$ such that for all $\xi\in\R$, 
$$
\sum_{\lambda\in\Lambda} |\widehat{\mu}(\xi+\lambda)|^2=\frac1n \sum_{s=0}^{n-1} |M_{{\mathcal B}_s}(\xi)|^2.
$$
\end{theorem}

\medskip

Let us prove our main theorem in the introduction assuming the above. 

\medskip

\noindent{\it Proof of Theorem \ref{theorem_main1} assuming Theorem \ref{thm_main1}}. As we have discussed in Subsection 2.1, we can without loss of generality assume $\D$ satisfies (\ref{eqD_1}) with all the stated properties. We now let 
$$
\Lambda_0= \frac{1}{N}{\mathcal L}_2 +\Lambda.
$$
Then applying Theorem \ref{thm_main1} and interchanging summation, 
$$
\sum_{\ell_2\in{\mathcal L}_2}\sum_{\lambda\in\Lambda} \left|\widehat{\mu}\left(\xi+\frac1N\ell_2+\lambda\right)\right|^2 = \frac1n\sum_{s=0}^{n-1} \sum_{\ell_2\in{\mathcal L}_2}\left|M_{{\mathcal B}_s}\left(\xi+\frac1N\ell_2\right)\right|^2. 
$$
As $\frac{1}{N}{\mathcal L}_2$ is a spectrum of ${\mathcal B}_s$ for all $s=0,...,n-1$, the above sum is equal to $1$.   This shows that Theorem \ref{theorem_main1} holds true.\qquad$\Box$

\medskip
The next two sections will be devoted to prove Theorem \ref{thm_main1}. The framework of the proof is based on the previous proofs for the standard Hadamard triple. However, certain careful and non-trivial modifications will be needed as now the exponential set $E(\Z)$ is not a complete set for $L^2(\mu)$. Let us summarize the prcedure as below.

\begin{enumerate}
    \item We will first investigate the identity for all finite level iterations with a large class of potential spectrum $\Lambda$.
    \item Then we will introduce a tail-term condition ensuring the limit-taking process will go through. 
    \item Finally, we study a modified periodic zero set to ensure the tail-term condition is always satsifed by some choice of $\Lambda$. 
\end{enumerate} 

\begin{remark}
Before we go into the proof, one may have noticed a seemingly easy approach.  The observation is that $\mu_{N,\D}$ is a convolution of two measures $\nu_1= \mu_{N^2,\D}$ and $\nu_2 (\cdot) = \nu_1 (N (\cdot))$. i.e. 
$$
\mu_{N,\D} = \nu_1\ast\nu_2. 
$$
Using Proposition \ref{proposition_hada-k} that we will later prove,  $(N^2, \D,N{\mathcal L}_1\oplus{\mathcal L}_2)$ is a Hadamard triple. Then by the result of {\L}aba and Wang, the measure $\nu_1$ is a spectral measure and so is $\nu_2$.  

\medskip

Despite having such a simple convolutional structure, due to the lack of periodicity on the Fourier transform of the measures, it is not possible to conclude $\mu_{N,\D}$ is a spectral measure with the spectrum formed by the Minkowski sum of  the two individual spectra.  Indeed, it is known that there exists spectral measures $\nu_1,\nu_2$ whose convolutional product $\nu_1\ast\nu_2$ is NOT spectral, and $\nu_1\ast\nu_2$ is still a  equal-weighted self-similar measure as well as it is the normalized Hausdorff measure on the support (see e.g. \cite{DHL2014}). 
\end{remark}






\medskip


\medskip

\section{Finite level iteration and sufficient conditions}

To begin our proof, in this section, 
we first establish several  finite level identities and then introduce some sufficient conditions for the limit to pass through. 

\medskip

\subsection{Finite level identities} We first establish the identity which allows some more general periodic function. It will be useful for later purposes. 

\begin{Lem}\label{lemma3.0}
Let the product-form Hadamard triple  $(N, {\mathcal D}, {\mathcal L} = {\mathcal L}_1\oplus {\mathcal L}_2)$ be given as in (\ref{eqD_1}). For each $s = 0,1,...,n-1$, we let  $Q_s(\xi)$  be an integer periodic function and define 
$$
{\bf M}(\xi) = \sum_{s=0}^{n-1} e\left(-\frac{a_s\xi}{N}\right) Q_s(\xi).
$$
Then 
$$
\sum_{\ell\in {\mathcal L}}|{\bf M}(\xi+\ell)|^2|M_{\frac{{\mathcal B}_s}{N}}(\xi+\ell)|^2  =  \frac{1}{n} \sum_{s=0}^{n-1}\left|Q_s (\xi)\right|^2.
$$
\end{Lem}
 \begin{proof}
For all $\ell \in{\mathcal L}$,
\begin{align}\label{eq-distr}
\left|{\bf M} (\xi+\ell)\right|^2\left|M_{\frac{{\mathcal B}_s}{N}} (\xi+\ell)\right|^2 =& \left(\frac1{n^2}\sum_{s_1, s_2=0}^{n-1}e\left(-\frac{(a_{s_1}-a_{s_2})(\xi+\ell)}{N}\right) Q_{s_1} (\xi+\ell)\overline{Q_{s_2}} (\xi+\ell)\right)\nonumber\\
&\cdot\left(\frac1{m^2}\sum_{j_1,j_2=0}^{m-1} e\left(-\frac{(b_{j_1}-b_{j_2})(\xi+\ell)}{N}\right)\right) \nonumber\\
=& \frac1{n^2m^2}\sum_{s_1, s_2=0}^{n-1}\sum_{j_1, j_2=0}^{m-1} e\left(-\frac{(a_{s_1}+b_{j_1}-a_{s_2}-b_{j_2})(\xi+\ell)}{N}\right)Q_{{s_1}}(\xi)\overline{{Q_{s_2}}} (\xi).
\end{align}
By condition (ii) of the product-form Hadamard triple  in Definition \ref{definition_product-form}.  $(N, {\mathcal A}\oplus{\mathcal B}_s, {\mathcal L})$ forms  a Hadarmard triple for all $s$, we have 
\begin{align*}
&\frac{1}{mn}\sum_{\ell\in {\mathcal L}}e\left(-\frac{(a_{s_1}+b_{j_1}-a_{s_2}-b_{j_2})(\xi+\ell)}{N}\right)\\
=&e\left(-\frac{(a_{s_1}+b_{j_1}-a_{s_2}-b_{j_2})\xi}{N}\right)M_{\frac{{\mathcal L}}{N}}(a_{s_1}+b_{j_1}-a_{s_2}-b_{j_2}) \\
=&\delta_0(s_1-s_2)\delta_0(j_1-j_2),
\end{align*}
where $\delta_0$ is the Dirac measure at $0$. It implies that
$$
\begin{aligned}
&\sum_{\ell\in {\mathcal L}}|{\bf M}(\xi+\ell)|^2|M_{\frac{{\mathcal B}_s}{N}}(\xi+\ell)|^2\\
=&\frac{1}{n^2m^2} \sum_{s_1, s_2=0}^{n-1}\sum_{j_1, j_2=0}^{m-1}Q_{s_1} (\xi)\overline{Q_{s_2}} (\xi) \sum_{\ell\in{\mathcal L}} e\left(-\frac{(a_{s_1}+b_{j_1}-a_{s_2}-b_{j_2})(\xi+\ell)}{N}\right) \\
=& \frac{1}{nm} \sum_{s=0}^{n-1}\sum_{j=0}^{m-1}\left|Q_s(\xi)\right|^2\\
=&\frac{1}{n} \sum_{s=0}^{n-1}\left|Q_s(\xi)\right|^2.
\end{aligned}
$$
This completes the proof.
\end{proof}

We now  establish the identity in other finite level iterations. Here, we let 
$$
\Gamma_p= {\mathcal L}+N{\mathcal L}+...+N^{p-1}{\mathcal L}.
$$
$\Gamma_p$ are in  distinct  residue classes modulo $N^p$ since those of ${\mathcal L}$ are in distinct classes modulo $N$, by condition (ii) in Definition \ref{definition_product-form}. The set of points  $\widetilde{\Gamma}_p$ such that 
$\widetilde{\Gamma}_{p}\equiv \Gamma_p$ (mod $N^p$) collects $(\#{\mathcal L})^p$ points where each of the points is congruent modulo  $N^p$ to exactly one element in $\Gamma_p$. All such $\widetilde{\Gamma}_{p}$ will satisfy the following lemma. 

\begin{Lem}\label{equi-distibution lemma}
Let $s\in\{0,1,...,n-1\}$. Then for all $p\ge 1$ and $\widetilde{\Gamma}_p\equiv\Gamma_p\pmod {N^p}$,
$$
\sum_{\lambda\in\widetilde{\Gamma_p}}\left|\widehat{\mu}_p(\xi+\lambda)\right|^2\left|M_{\frac{{\mathcal B}_s}{N^p}}(\xi+\lambda)\right|^2=\frac1n \sum_{i=0}^{n-1} \left|M_{{\mathcal B}_i}(\xi)\right|^2
$$
\end{Lem}

\medskip

\begin{proof} Recall that $\widehat{\mu}_p(\xi) = \prod_{j=1}^p M_{\D}(N^{-p}\xi)$. As $\D$ and ${\mathcal B}_s$ are all integer sets,  $\widehat{\mu}_p(\xi)$ and $M_{\frac{{\mathcal B}_s}{N^p}}(\xi)$  are $N^p-$period, we have
$$\sum_{\lambda\in\widetilde{\Gamma_p}}\left|\widehat{\mu}_p(\xi+\lambda)\right|^2\left|M_{\frac{{\mathcal B}_s}{N^p}}(\xi+\lambda)\right|^2=\sum_{\lambda\in{\Gamma_p}}\left|\widehat{\mu}_p(\xi+\lambda)\right|^2\left|M_{\frac{{\mathcal B}_s}{N^p}}(\xi+\lambda)\right|^2.$$
It suffices to prove the identity for $\Gamma_p$. We prove the result by induction. When $p=1$, $\mu_p=\delta_{{\mathcal D}/N}$ and $\Gamma_p= {\mathcal L}$. Note that 
$$
M_{\frac{{\mathcal D}}{N}} (\xi) = \frac{1}{n}\sum_{s=0}^{n-1} e\left(-\frac{a_s\xi}{N}\right) M_{{\mathcal B}_s}(\xi).
$$
Applying Lemma \ref{lemma3.0} with $Q_s(\xi) = M_{{\mathcal B}_s}(\xi)$, we have 
$$
\begin{aligned}
\sum_{\ell\in {\mathcal L}}|M_{\frac{{\mathcal D}}{N}}(\xi+\ell)|^2|M_{{\frac{{\mathcal B_s}}{N}}}(\xi+\ell)|^2
=&\frac{1}{n} \sum_{i=0}^{n-1}\left|M_{\mathcal B_i} (\xi)\right|^2.\\
\end{aligned}
$$
Suppose that the assertion is true for $p-1$. We decompose $\Gamma_p= \Gamma_{p-1}+ N^{p-1}{\mathcal L}$. Note that $\mu_{p-1}$ is $N^{p-1}$-periodic. Now our proof of lemma is completed after the following calculation. 
$$
\begin{aligned}
\sum_{\lambda\in\Gamma_p}|\widehat{\mu}_p(\xi+\lambda)|^2|M_{\frac{{\mathcal B}_s}{N^p}}(\xi+\lambda)|^2=& \sum_{\lambda\in\Gamma_{p-1}}|\widehat{\mu}_{p-1}(\xi+\lambda)|^2\sum_{\ell\in{\mathcal L}}\left|M_{\frac{{\mathcal D}}{N}}\left(\frac{\xi+\lambda}{N^{p-1}}+ \ell\right)\right|^2\left|M_{\frac{{\mathcal B}_s}{N}}\left(\frac{\xi+\lambda}{N^{p-1}}+\ell\right)\right|^2 \\
=&\sum_{\lambda\in\Gamma_{p-1}}|\widehat{\mu}_{p-1}(\xi+\lambda)|^2\left(\frac1n\sum_{j=0}^{n-1}|M_{\frac{{\mathcal B}_j}{N^{p-1}}} (\xi+\lambda)|^2\right) \ (\mbox{by the case $p=1$})\\
=& \frac1n\sum_{j=0}^{n-1}\sum_{\lambda\in\Gamma_{p-1}}\left|\widehat{\mu}_{p-1}(\xi+\lambda)\right|^2\left|M_{\frac{{\mathcal B}_j}{N^{p-1}}} (\xi+\lambda)\right|^2\\
=& \frac1n\sum_{j=0}^{n-1} \frac1n\sum_{i=0}^{n-1}|M_{{\mathcal B}_i}(\xi)|^2 \ (\mbox{by induction hypothesis})\\
=&\frac1n\sum_{i=0}^{n-1}|M_{{\mathcal B}_i}(\xi)|^2.
\end{aligned}
$$
\end{proof}

\bigskip

Heuristically, if we can take limit as $p\to \infty$ in the above lemma, we will obtain Theorem \ref{thm_main1} immediately. However, even in the case of Hadamard triple,   this is not going to work out directly. Extra assumption on the $\Lambda$ is needed. In particular, we need a tail-term condition in (\ref{tail-term}).

\medskip

\subsection{Tail-term condition.} Given a subsequence of positive integers $\{p_k\}$, letting $q_0 = 0$ and  $q_k=p_1+\cdots+p_k$.
We can define inductively an orthogonal set by $\Lambda_0 = \{0\}$.
$$
\Lambda_{q_k}= \widetilde{\Gamma}_{p_1} +N^{q_{1}}\widetilde{\Gamma}_{p_2}+\cdots+ N^{q_{k-1}}\widetilde{\Gamma}_{p_k},  
$$
where $0\in\widetilde{\Gamma}_{p_j}$ and $\widetilde{\Gamma}_{p_j}\equiv \Gamma_{p_j}(\mbox{mod} \ N^{p_k})$ for all $1\le j\le k$. Then $\Lambda_{q_k}\subset\Lambda_{q_{k+1}}$. Let
\begin{equation}\label{eq_Lambda_candidate}
\Lambda = \bigcup_{k=1}^{\infty} \Lambda_{q_k}.    
\end{equation}

\medskip

\begin{Prop}\label{lem-main}
For any $\Lambda$ defined in (\ref{eq_Lambda_candidate}), we have 
\begin{enumerate}
    \item $$\sum_{\lambda\in\Lambda} |\widehat{\mu}(\xi+\lambda)|^2 \le \frac{1}{n} \sum_{s=0}^{n-1} |M_{\B_s}(\xi)|^2$$
    \item Suppose that there is a positive constant $c>0$ such that 
    \begin{equation}\label{tail-term}
   |\widehat{\mu}_{>q_k} (\xi+\lambda)|^2 \ge \frac{c}{n} \sum_{s=0}^{n-1} |M_{\frac{\B_s}{N^{q_k}}}(\xi+\lambda)|^2,\quad \forall \ \lambda\in\Lambda_{k}, k\geq1,   \ \forall \xi\in[0,1]  
    \end{equation}
    Then $$\sum_{\lambda\in\Lambda} |\widehat{\mu}(\xi+\lambda)|^2 = \frac{1}{n} \sum_{s=0}^{n-1} |M_{\B_s}(\xi)|^2$$
\end{enumerate}
\end{Prop}

\begin{proof}
\rm{(i)} From the definition of $\Lambda_{k}$, 
$$\Lambda_{k}\equiv \Gamma_{q_k}\quad \pmod{N^{q_k}}.$$
Note that for each $k\geq1$, we have 
$$|\mu(\xi)|^2\le |\mu_{q_k+1}(\xi)|^2= |\mu_{q_k}(\xi)|^2\left|M_{\frac{\D}{N^{q_k}+1}}(\xi)\right|^2,$$
and 
$$
\begin{aligned}
\left|M_{\frac{\D}{N^{q_k}+1}}(\xi)\right|^2=&\left|\frac1n\sum_{s=0}^{n-1}e\left(-\frac{a_s\xi}{N^{q_k+1}}\right)M_{\frac{{\mathcal B}_s}{N^{q_k}}}(\xi)\right|^2\\
\le&\left(\frac1n\sum_{s=0}^{n-1}|M_{\frac{{\mathcal B}_s}{N^{q_k}}}(\xi)|\right)^2\\
\le&\frac1n\sum_{s=0}^{n-1}|M_{\frac{{\mathcal B}_s}{N^{q_k}}}(\xi)|^2.
\end{aligned}
$$
So for any $k\geq1$, 
$$
\begin{aligned}
\sum_{\lambda\in\Lambda_{q_k}} |\widehat{\mu}(\xi+\lambda)|^2\le&  \sum_{\lambda\in\Lambda_{q_k}} |\widehat{\mu}_{q_k}(\xi+\lambda)|^2\left|M_{\frac{\D}{N^{q_k+1}}}(\xi+\lambda)\right|^2\\
\le&\frac{1}{n} \sum_{s=0}^{n-1} \sum_{\lambda\in\Lambda_{q_k}} |\widehat{\mu}_{q_k}(\xi+\lambda)|^2\left|M_{\frac{{\mathcal B}_s}{N^{q_k}}}(\xi+\lambda)\right|^2\\
=&\frac{1}{n} \sum_{s=0}^{n-1} \frac{1}{n} \sum_{i=0}^{n-1} \left|M_{{{\mathcal B}_i}}(\xi)\right|^2\quad (\text{using Lemma \ref{equi-distibution lemma}})\\
=&\frac{1}{n} \sum_{s=0}^{n-1} \left|M_{{{\mathcal B}_s}}(\xi)\right|^2,\\
\end{aligned}
$$
Letting $k\to \infty$, we have 
$$\sum_{\lambda\in\Lambda} |\widehat{\mu}(\xi+\lambda)|^2\le\frac{1}{n} \sum_{s=0}^{n-1} \left|M_{{{\mathcal B}_s}}(\xi)\right|^2.$$

\rm{(ii)}  Let us write $$\Psi_{q_k}(\xi)=\sum_{\lambda\in\Lambda_{q_k}}\left|\widehat{\mu}(\xi+\lambda)\right|^2\quad \text{and}\quad \Psi_\Lambda(\xi)=\sum_{\lambda\in\Lambda}\left|\widehat{\mu}(\xi+\lambda)\right|^2.$$ 
It has been proved in \rm{(i)} that 
$$\Psi_\Lambda(\xi)\le \frac1n\sum_{s=0}^{n-1}|M_{{{\mathcal B}_s}}(\xi)|^2.$$
For any $k,t\geq1$, we have the following identity:
\begin{eqnarray}\label{eq2.1}
\Psi_{q_{k+t}}(\xi)&=&\Psi_{q_k}(\xi)+\sum_{\lambda\in\Lambda_{q_{k+t}}\setminus\Lambda_{q_k}}\left|\widehat{\mu}(\xi+\lambda)\right|^2\nonumber\\
 &=&\Psi_{q_k}(\xi)+\sum_{\lambda\in\Lambda_{q_{k+t}}\setminus\Lambda_{q_k}}\left|\widehat{\mu}_{q_{n+t}}(\xi+\lambda)\right|^2
 \left|\widehat{\mu}_{>q_{n+t}}(\xi+\lambda)\right|^2\nonumber\\
 &\geq&\Psi_{q_k}(\xi)+\frac{c}{n}\sum_{s=0}^{n-1}\sum_{\lambda\in\Lambda_{q_{k+t}}\setminus\Lambda_{q_k}}\left|\widehat{\mu}_{q_{k+t}}(\xi+\lambda)\right|^2|M_{\frac{{\mathcal B}_s}{N^{q_{k+t}}}}(\xi+\lambda)|^2\nonumber\\
&=&\Psi_{q_k}(\xi)+\frac{c}{n}\sum_{s=0}^{n-1}\left(\frac1n\sum_{j=0}^{n-1}|M_{{{\mathcal B}_j}}(\xi)|^2-\sum_{\lambda\in\Lambda_{{q_k}}}\left|\widehat{\mu}_{q_{k+t}}(\xi+\lambda)\right|^2|M_{\frac{{\mathcal B}_s}{N^{q_{k+t}}}}(\xi+\lambda)|^2\right) \nonumber\\
&=&\Psi_{q_k}(\xi)+c\left(\frac1n\sum_{s=0}^{n-1}|M_{{{\mathcal B}_s}}(\xi)|^2-\frac{1}{n}\sum_{s=0}^{n-1}\sum_{\lambda\in\Lambda_{{q_k}}}\left|\widehat{\mu}_{q_{k+t}}(\xi+\lambda)\right|^2|M_{\frac{{\mathcal B}_s}{N^{q_{k+t}}}}(\xi+\lambda)|^2\right).
\end{eqnarray}
Letting $t\to+\infty$, we have
\begin{eqnarray*}
\Psi_{\Lambda}(\xi)&\geq& \Psi_{q_k}(\xi)+c\left(\frac1n\sum_{s=0}^{n-1}|M_{{{\mathcal B}_s}}(\xi)|^2-\frac{1}{n}\sum_{s=0}^{n-1}\Psi_{q_k}(\xi)\right)\\
&=&\Psi_{q_k}(\xi)+c\left(\frac1n\sum_{s=0}^{n-1}|M_{{{\mathcal B}_s}}(\xi)|^2-\Psi_{q_k}(\xi)\right).
\end{eqnarray*}
And letting $k\to+\infty$, we have
$$\Psi_{\Lambda}(\xi)\geq \Psi_{\Lambda}(\xi)+c\left(\frac1n\sum_{s=0}^{n-1}|M_{{{\mathcal B}_s}}(\xi)|^2-\Psi_{\Lambda}(\xi)\right).$$
So $\Psi_{\Lambda}(\xi)\geq\frac1n\sum_{s=0}^{n-1}|M_{{{\mathcal B}_s}}(\xi)|^2$, which together with  \rm{(i)}  imply the result.
\end{proof}

\subsection{Equi-positivty condition}

Condition (\ref{tail-term}) gives a control on the tail-term of the Fourier transform of $\widehat{\mu}$. Because of the self-similarity, we can formulate it as the following equi-positivity condition. Such a condition was studied previously in \cite{DHL2019}. In our situation, it needs to be readjusted as an average condition.

\begin{Def}
Let $\D=\D_1$ be the product-form defined in (\ref{eq_product-form_Hada}). 
We say that the probability measure $\mu = \mu_{N,\D}$ satisfies the {\it  average equi-positive condition} if there exists  $c>0$ and $\delta>0$  such that for all $\xi\in [0, 1]$, there exists $k_{\xi}\in\Z$ such that 
\begin{equation}\label{eq-main-1}
|\widehat{\mu}(\xi+y+k_{\xi})|^2\geq \frac{c}{n}\sum_{s=0}^{n-1}|M_{{\mathcal B_s}}(\xi+y)|^2
\end{equation}
whenever $|y|<\delta$. 
\end{Def}

\medskip

The average equi-positive condition is sufficient to construct an $\Lambda$ verifying the tail-term condition \eqref{tail-term}, leading to the following theorem. 

\begin{theorem}\label{th-main}
Let $(N, {\mathcal D}, {\mathcal L})$ be a product-form Hadamard triple and the average equi-positive condition is satisfied.  Then there is a $\Lambda\subset\Z$ such that for all $\xi\in[0, 1]$, 
$$
\sum_{\lambda\in\Lambda} |\widehat{\mu}(\xi+\lambda)|^2=\frac1n \sum_{s=0}^{n-1} |M_{{\mathcal B}_s}(\xi)|^2.
$$
\end{theorem}
\begin{proof} Since   $(N, {\mathcal D}_0, {\mathcal L})$ forms a Hadamard triple, the elements of ${\mathcal L}$ are in distinct residue modulo $N$. So, without loss of generality, we can assume $0\in{\mathcal L}\subset\{0, 1, \cdots, N-1\}.$
We now construct inductively $\Lambda_{q_k}$ so that  
\begin{equation}\label{eq-main}
  |\widehat{\mu}_{>q_k} (\xi+\lambda_k)|^2=\left|\widehat{\mu} \left(\frac{\xi+\lambda_k}{N^{q_k}}\right)\right|^2 \ge \frac{c}{n} \sum_{i=0}^{n-1} \left|M_{\frac{\B_i}{N^{q_k}}}(\xi+\lambda_k)\right|^2,\quad \forall\ \xi\in[0,1],\  \lambda_k\in\Lambda_{q_k},  
\end{equation}
    then the result follows from Lemma \ref{lem-main}. First, we take $\Lambda_0=\{0\}$. Suppose that $\Lambda_{q_{k-1}}$ has been constructed which satisfies the inequality \eqref{eq-main}.  We can take a large enough $p_k$ in the subsequence so that 
    \begin{equation}\label{eq-main-2}
    \sup\left\{|N^{-q_k}(\xi+\lambda_{k-1})|: \xi\in[0, 1],\lambda_{k-1}\in\Lambda_{q_{k-1}}\right\}<\delta.
    \end{equation}
We now define 
$$\widetilde{\Gamma}_{p_{k}}=\{\gamma_k+N^{p_k}k_{\xi_{\gamma_k}}: \gamma_k\in\Gamma_{p_k}\} \text{ and } \Lambda_{q_k}=\Lambda_{q_{k-1}}+N^{q_{k-1}}\widetilde{\Gamma}_{p_{k}},$$
where $\xi_{\gamma_k}=\gamma_k/N^{p_k}\in[0, 1]$ so that $k_{\xi_{\gamma_k}}\in\Z$ is defined as in \eqref{eq-main-1}. Now, writing 
$$\lambda_k=\lambda_{k-1}+N^{q_{k-1}}\gamma_k+N^{q_{k}}k_{\xi_{\gamma_k}},$$
for some $\lambda_{k-1}\in\Lambda_{q_{k-1}}$, we have
\begin{eqnarray*}
 |\widehat{\mu}_{>q_k} (\xi+\lambda_k)|^2&=&\left|\widehat{\mu} \left(\frac{\xi+\lambda_k}{N^{q_k}}\right)\right|^2\\
 &=&\left|\widehat{\mu} \left(\frac{\xi+\lambda_{k-1}+N^{q_{k-1}}\gamma_k}{N^{q_k}}+k_{\xi_{\gamma_k}}\right)\right|^2\quad (\text{using \eqref{eq-main-2}})\\
 &\geq& \frac{c}{n}\sum_{i=0}^{n-1}\left|M_{B_i}(\frac{\xi+\lambda_{k-1}+N^{q_{k-1}}\gamma_k}{N^{q_k}})\right|^2\\
 &=&\frac{c}{n}\sum_{i=0}^{n-1}\left|M_{\frac{B_i}{N^{q_k}}}(\xi+\lambda_{k})\right|^2.
\end{eqnarray*}
We have completed the proof.
\end{proof}

\bigskip

\bigskip

\section{Proof of Theorem \ref{thm_main1}}

In this section, we will prove that the average equi-positive condition holds for our product-form Hadamard triple. Then Theorem \ref{th-main} will imply Theorem \ref{thm_main1}. Recall that 
$$|\widehat{\mu}(\xi)|=\left|M_{\frac{\D}{N}}(\xi)\right|\cdot\left|\widehat{\mu}(\frac{\xi}{N})\right|.$$
We will study the Fourier transform in the two parts. In the first part,
$$
M_{\frac{\D}{N}}(\xi) =\frac1n \sum_{i=0}^{n-1} e\left(-\frac{a_i}{N}\xi\right) M_{\B_i}(\xi). 
$$
Unfortunately, a common zero for all $\B_s$ will create an obstacle in obtaining a slight perturbation $|y|<\delta$ in the definition. To overcome the situation, we need to introduce the following factorization procedure for the common zeros. Let $\B$ be a finite subset of non-negative integers, we define 
$$
P_{\B}(z) = \sum_{b\in\B}z^b
$$
which belongs to $\Z[x]$, the ring of  polynomials with integer coefficients.
Then 
$$
M_{\B}(\xi) = \frac{1}{\#B} P_{\B}(e(-\xi)). 
$$
Note that if $e(-\theta) = e^{-2\pi i \theta}$ is a zero for $P_{\B}$, then $e(-\theta)$ is an algebraic number and irreducible polynomial $F_{\theta}(z)\in\Z[x]$ for $e(-\theta)$ is the smallest degree polynomial having $e(-\theta)$ as a root. We know that it divides $P_{\B}$. We let 
$$
Z = \{e(-\theta): P_{\B_s}(e(-\theta)) = 0, \forall s = 0, 1, ..., n-1\}.  
$$
Note that $Z$ is a finite set and for each $e(-\theta)$ and $s = 0,1..,n-1$, let $k_{\theta, s} = \max\{k: F_{\theta}^{k}(x) \ \mbox{divides} \ P_{\B_s}(x)\}$. Then 
$$
k_{\theta} = \min \{k_{\theta,s}: s= 0,1..,n-1\}. 
$$
From this construction, the polynomial 
$$
F(x) = \prod_{e(-\theta)\in Z} F_{\theta}^{k_{\theta}}(x)
$$ 
divides all $P_{s}(x)$, so for all $s = 0,...,n-1$, 
$$
P_{\B_s}(x) = F(x)Q_s(x)
$$
for some $Q_s(x)\in\Z[x]$. As common zeros are now all factorized, the function 
\begin{equation}\label{eq_QQ}
{\bf Q}(\xi) := \frac1{n} \sum_{s=0}^{n-1}\left|Q_s(e(-\xi))\right|^2 >0.     
\end{equation}
We now consider the following derived function
$$
\widehat{M_{\frac{\D}{N}}}(\xi) = \frac{1}{n}\sum_{s=0}^{n-1} e\left(-\frac{a_s\xi}{N}\right) Q_s(e(-\xi))
$$
Notice that 
\begin{equation}\label{eq_factorize}
    M_{\frac{\D}{N}}(\xi) = F(e(-\xi))\cdot \widehat{M_{\frac{\D}{N}}}(\xi) \ \mbox{and} \ M_{\B_s}(\xi) = F(e(-\xi))\cdot Q_s(e(-\xi)). 
\end{equation}
Moreover, as the polynomial $F(x)$ and $Q(x)$ are in $\Z[x]$, we have 
\begin{equation}\label{eq_factorize2}
   F(e(-\xi-\ell))= F(e(-\xi)),\quad  Q_s(e(-\xi-\ell)) = Q_s(e(-\xi)) 
\end{equation}
for all $\ell\in\Z$ and $s = 0,1,...,n-1$. Note that all $\B_s$ have the same cardinality since $(N,\B_s)$ are equivalent compatible pairs, we will denote it by $m = \#\B_s$.  

\begin{Lem}\label{lem-positive}
There exists $c>0$ and $\delta>0$ such that for all $\xi\in[0, 1]$, there is a $\ell_{\xi}\in{\mathcal L}$ such that 
and 
\begin{equation}\label{eq_lemma4.1-1}
    \left|M_{\frac{\D}{N}}(\xi+y+\ell_{\xi})\right|^2\geq \frac{1}{2mn^2}\sum_{s=0}^{n-1}|M_{B_s}(\xi+y)|^2
\end{equation}
and \begin{equation}\label{eq_lemma4.1-2}\frac1n\sum_{s=0}^{n-1}\left|M_{\frac{B_{s}}{N}}({\xi+y+\ell_{\xi}})\right|^2\geq \frac{1}{2m^2n}.
\end{equation}
whenever $|y|<\delta$. 
In particularly, we can take $\ell_{\xi}=0$ if $\xi=0$.
\end{Lem}

\begin{proof} Let $\xi\in[0,1]$. Applying Lemma \ref{lemma3.0}, we have
\begin{eqnarray*}
&&\sum_{\ell\in{\mathcal L}}\left|\widehat{M_{\frac{\D}{N}}}(\xi+\ell)\right|^2\left(\frac1n\sum_{s=0}^{n-1}\left|M_{\frac{B_s}{N}}(\xi+\ell)\right|^2\right)\nonumber\\
&=&\frac1n\sum_{s=0}^{n-1}\sum_{\ell\in{\mathcal L}}\left|\widehat{M_{\frac{\D}{N}}}(\xi+\ell)\right|^2\left|M_{\frac{B_s}{N}}(\xi+\ell)\right|^2\nonumber\\
&=&\frac{1}{n}\sum_{s=0}^{n-1}|Q_s(e(-\xi))|^2.
\end{eqnarray*}
Hence,  there is a $\ell_{\xi}\in{\mathcal L}$ such that  
\begin{eqnarray}\label{eq-positive1}
\left|\widehat{M_{\frac{\D}{N}}}(\xi+\ell_{\xi})\right|^2&\geq&\left|\widehat{M_{\frac{\D}{N}}}(\xi+\ell_{\xi})\right|^2\left(\frac1n\sum_{s=0}^{n-1}\left|M_{\frac{B_s}{N}}(\xi+\ell_{\xi})\right|^2\right)\nonumber\\
&\geq&\frac{1}{mn^2}\sum_{s=0}^{n-1}\left|Q_s(e(-\xi))\right|^2:= \frac{{\bf Q}(\xi)}{mn}>0. 
\end{eqnarray} 
The strict positivity follows from (\ref{eq_QQ}). Now, as ${\bf Q}>0$,  the function for $\ell\in{\mathcal L}$
$$
{\bf T}_{\ell}(x): = \frac{\left|\widehat{M_{\frac{\D}{N}}}(x+\ell)\right|^2}{{\bf Q}(x)}
$$
is a continuous function on $[0,1]$ and hence is uniformly continuous. Moreover, as there are only finitely many such functions, $\{{\bf T}_{\ell}: \ell\in{\mathcal L}\}$ are equi-continuous on $[0,1]$. Hence,  we can find $\delta>0$ (independent of $\ell\in{\mathcal L}$) such that for all $|x-y|<\delta$.
$$
|{\bf T}_{\ell}(x)-{\bf T}_{\ell}(y)| < \frac{1}{2mn}. 
$$
As ${\bf T}_{\ell_{\xi}} (\xi) \ge \frac{1}{mn}$ by (\ref{eq-positive1}),  we have if $|y|<\delta$,
$$
{\bf T}_{\ell_{\xi}} (\xi+y) \ge \frac{1}{2mn}.
$$
Hence, 
$$
\left|\widehat{M_{\frac{\D}{N}}}(\xi+y+\ell_{\xi})\right|^2\ge \frac{1}{2mn^2}\sum_{s=0}^{n-1}|Q_s(e(-\xi-y))|^2.
$$
To finish the proof of (\ref{eq_lemma4.1-1}), we just need to multiply $|F(e(-\xi-y))|^2$ on both sides of the inequality and use (\ref{eq_factorize}) and (\ref{eq_factorize2}). This completes the proof of (\ref{eq_lemma4.1-1}).

\medskip

We now establish (\ref{eq_factorize2}). Since ${\mathcal L}_1$ is a spectrum of ${\mathcal A}$, we have
\begin{eqnarray*}
\sum_{\ell\in{\mathcal L}}\left|\widehat{M_{\frac{\D}{N}}}(\xi+\ell)\right|^2
&=&\sum_{\ell_2\in{\mathcal L}_2}\sum_{\ell_1\in{\mathcal L}_1}\left|\widehat{M_{\frac{\D}{N}}}(\xi+\ell_1+\ell_2)\right|^2\\
&=&\frac{1}{n^2}\sum_{\ell_2\in{\mathcal L}_2}\sum_{\ell_1\in{\mathcal L}_1}\sum_{s_1, s_2=0}^{n-1}e\left(-\frac{(a_{s_1}-a_{s_2})(\xi++\ell_2+\ell_1)}{N}\right)Q_{s_1}(e(-\xi))\overline{Q_{s_2}(e(-\xi))}\\
&=&\frac{1}{n^2}\sum_{\ell_2\in{\mathcal L}_2}\sum_{s_1, s_2=0}^{n-1}e\left(-\frac{(a_{s_1}-a_{s_2})(\xi+\ell_2)}{N}\right)Q_{s_1}(\xi)\overline{Q_{s_2}(\xi)}\sum_{\ell_1\in{\mathcal L}_1}e\left(-\frac{(a_{s_1}-a_{s_2})\ell_1}{N}\right)\\
&=&\frac{m}{n}\sum_{s=0}^{n-1}\left|Q_s(e(-\xi))\right|^2\\
\end{eqnarray*}
Hence,
$$\frac{m}{n}\sum_{s=0}^{n-1}\left|Q_s(e(-\xi))\right|^2\geq\left|\widehat{M_{\frac{\D}{N}}}(\xi+\ell_{\xi})\right|^2,$$
which together with \eqref{eq-positive1} implies that 
\begin{eqnarray*}
\left|\widehat{M_{\frac{\D}{N}}}(\xi+\ell_{\xi})\right|^2\left(\frac1n\sum_{s=0}^{n-1}\left|M_{\frac{B_s}{N}}(\xi+\ell_{\xi})\right|^2\right)
&\geq&\frac{1}{mn^2}\sum_{s=0}^{n-1}|Q_s(e(-\xi))|^2\\
&\geq& \frac{1}{m^2n}\left|\widehat{M_{\frac{\D}{N}}}(\xi+\ell_{\xi})\right|^2.
\end{eqnarray*}
From (\ref{eq-positive1}), $\left|\widehat{M_{\frac{\D}{N}}}(\xi+\ell_{\xi})\right|^2>0$, we thus obtain 
$$
\frac1n\sum_{s=0}^{n-1}\left|M_{\frac{B_s}{N}}(\xi+\ell_{\xi})\right|^2 \ge\frac{1}{m^2n}.
$$
Now, the functions $\{{\bf S}_{\ell}(x) = \frac1n\sum_{s=0}^{n-1}\left|M_{\frac{B_s}{N}}(x+\ell)\right|^2:\ell\in {\mathcal L}\}$ is equicontinuous on $[0,1]$, the same argument established (\ref{eq_factorize2}).

\end{proof}
\bigskip

\subsection{Weakly periodic set.} Let  ${\mathcal O}=\left\{\xi\in[0, 1]: \frac1n\sum_{s=0}^{n-1}|M_{{\mathcal B}_s}(\xi)|^2>0\right\}$. We define the {\it weakly periodic set} of $\mu$ to be 
$${\mathcal W}(\widehat{\mu})=\{\xi\in{\mathcal O}: \widehat{\mu}(\xi+k)=0,\ k\in\Z\}.$$
The following corollary is immediate. 

\bigskip

\begin{corollary}\label{coro-positive}
For any  $\xi\in {\mathcal O}$, there is a $\ell_{\xi}\in{\mathcal L}$ such that 
$$\left|M_{\frac{\D}{N}}(\xi+\ell_{\xi})\right|^2\left(\frac1n\sum_{s=0}^{n-1}\left|M_{\frac{B_s}{N}}(\xi+\ell_{\xi})\right|^2\right)>0$$
\end{corollary}
\begin{proof}
It follows directly from Lemma \ref{lem-positive}.
\end{proof}

\begin{Lem}\label{W-empty}
With the above notations, for the product-form Hadamard triple in Definition \ref{definition_product-form} with $r = 1$,  the associated self-similar measure $\mu=\mu_{N,\D}$ satisfies  ${\mathcal W}(\widehat{\mu})=\emptyset$.
\end{Lem}
\begin{proof}
Suppose that ${\mathcal W}(\widehat{\mu})\neq\emptyset$ and we take $\xi_0\in {\mathcal W}(\widehat{\mu})$. As $\widehat{\mu}(0)=1$, ${\mathcal W}(\widehat{\mu})\cap\Z=\emptyset$, so $\xi_0\not\in\Z$. We claim the following implication holds:

\begin{equation}\label{eq-claim}
\mbox{{\bf Claim:}} \    \left|M_{\frac{\D}{N}}(\xi_0+\ell)\right|^2\left(\frac1n\sum_{s=0}^{n-1}\left|M_{\frac{B_s}{N}}(\xi_0+\ell)\right|^2\right)>0\Longrightarrow \frac{\xi_0+\ell}{N}\in{\mathcal W}(\widehat{\mu}).
\end{equation}
Indeed, the assumption in the claim implies that $\frac{\xi_0+\ell}{N}\in {\mathcal O}$. By considering $k$ of the form $\ell+Nt$ and $t\in\Z$, we have
\begin{eqnarray*}
0=\widehat{\mu}(\xi_0+k)&=&M_{\frac{\D}{N}}(\xi_0+\ell+Nt)\widehat{\mu}\left(\frac{\xi_0+\ell+Nt}{N}\right)\\
&=&M_{\frac{\D}{N}}(\xi_0+\ell)\widehat{\mu}\left(\frac{\xi_0+\ell}{N}+t\right).
\end{eqnarray*}
As $M_{\frac{\D}{N}}(\xi_0+\ell)\neq0$, we must have $\widehat{\mu}(\frac{\xi_0+\ell}{N}+t)=0$ for all $t\in\Z$ and hence $\frac{\xi_0+\ell}{N}\in{\mathcal W}(\widehat{\mu})$. With this claim in mind, we define $Y_0=\{\xi_0\}$ and define inductively the set $Y_{k}$ by 
$$Y_k=\left\{\frac{\xi+\ell}{N}: \xi\in Y_{k-1}, \left|M_{\frac{\D}{N}}(\xi+\ell)\right|^2\left(\frac1n\sum_{s=0}^{n-1}\left|M_{\frac{B_s}{N}}(\xi+\ell)\right|^2\right)>0 \right\}.$$
By \eqref{eq-claim}, $Y_k\subset{\mathcal W}(\mu)$. From Corollary \ref{coro-positive}, we have all the sets $Y_k$ are non-empty. Also if $\xi_k\in Y_k$, then 
$$\xi_k=\frac{1}{N^k}(\xi_0+\ell_1+\cdots+N^{k-1}\ell_{k}).$$
This means for different $(\ell_1, \cdots, \ell_k)\neq(\ell_1', \cdots, \ell_k')$, the corresponding $\xi_k$ and $\xi_k'$ are different since the elements in ${\mathcal L}$ are in distinct residue modulo $N$. Therefore the cardinality of $Y_k$ is increasing. 
\smallskip

On ${\mathbb R}, \widehat{\mu}$ has only finitely many zeros in $[0, 1]$. Therefore, there exists $k_0$ such that for all $k\geq k_0$, the cardinality of $Y_{k}$ becomes a constant. This means that when $k\geq k_0$, each $\xi_k$ has only one offspring $\xi_{k+1}=\frac{1}{N}(\xi_k+\ell_{\xi_k})$, i.e. there is only one $\ell_{\xi_k}\in{\mathcal L}$ such that 
$$\left|M_{\frac{\D}{N}}(\xi_k+\ell_{\xi_k})\right|^2\left(\frac1n\sum_{s=0}^{n-1}\left|M_{\frac{\B_s}{N}}(\xi_k+\ell_{\xi_k})\right|^2\right)>0.$$
Recall that $\xi_k\in{\mathcal W}(\mu)\subset{\mathcal O}$, we have
\begin{eqnarray}\label{eq-W1}
0&<&\frac1n\sum_{s=0}^{n-1}\left|M_{\B_s}(\xi_k)\right|^2\nonumber\\
&=&\sum_{\ell\in{\mathcal L}}\left|M_{\frac{\D}{N}}(\xi_k+\ell)\right|^2\left(\frac1n\sum_{s=0}^{n-1}
\left|M_{\frac{\B_s}{N}}(\xi_k+\ell)\right|^2\right)\nonumber\\
&=&\left|M_{\frac{\D}{N}}(\xi_k+\ell_{\xi_k})\right|^2\left(\frac1n\sum_{s=0}^{n-1}\left|M_{\frac{\B_s}{N}}(\xi_k+\ell_{\xi_k})\right|^2\right)\nonumber\\
&=&\frac1n\sum_{s=0}^{n-1}\left| \frac1{n}\sum_{i=0}^{n-1} e\left(-\frac{a_{i}(\xi_k+\ell_{\xi_k})}{N}\right)M_{{\mathcal B}_s/N} (\xi_k+\ell_{\xi_k})M_{\mathcal B_{i}}(\xi_k)\right|^2\nonumber\\
&\le& \frac1n\sum_{s=0}^{n-1}\left(\frac1{n^2}\sum_{i=0}^{n-1}\left|e\left(-\frac{a_{i}(\xi_k+\ell_{\xi_k})}{N}\right)M_{{\mathcal B}_s/N} (\xi_k+\ell_{\xi_k})\right|^2\right)\left(\sum_{i=0}^{n-1}|M_{\mathcal B_{i}}(\xi_k)|^2\right)\nonumber\\
&\le&\frac{1}{n}\sum_{i=0}^{n-1}|M_{\mathcal B_{i}}(\xi_k)|^2.
\end{eqnarray}
This implies we have equality in a triangle inequality. From the second last line we have  
\begin{equation}\label{eq-W2}
1=M_{\frac{{\mathcal B}_s}{N}}(\xi_k+\ell_{\xi_k})=\frac1m\sum_{b\in{\mathcal B}_s}e\left(-\frac{b(\xi_k+\ell_{\xi_k})}{N}\right)
\end{equation}
where  $0\le s\le n-1$. Since $0\in{\mathcal B}_s$ (by our assummption after \eqref{eqD_1}), we get that $\frac{b(\xi_k+\ell_{\xi_k})}{N}\in\Z$ for all $b\in{\mathcal B}_s$ and so $b\xi_k\in\Z$. This forces
\begin{equation}\label{eq-W3}
M_{{{\mathcal B}_s}}(\xi_k)=\frac1m\sum_{b\in{\mathcal B}_s}e\left(-\xi_k{(b-b_s)}\right)=1,\quad \forall\ 0\le s\le n-1.  \end{equation}
Substituting \eqref{eq-W2} and \eqref{eq-W3} into \eqref{eq-W1}, we have
$$M_{\frac{\D_0}{N}}(\xi_k+\ell_{\xi_k})=\frac{1}{n}\sum_{i=0}^{n-1}e\left(-\frac{a_{i}(\xi_k+\ell_{\xi_k})}{N}\right)M_{\frac{{\mathcal B}_i}{N}}(\xi_k+\ell_{\xi_k})=1.$$
As $0\in\D_0$ and $\gcd(\D_0)=1$ (by our assummption after \eqref{eqD_1}), we have $\frac{\xi_k+\ell_{\xi_k}}{N}\in\Z$ and hence $\xi_k\in\Z$. This is a contradiction, since $\frac{\xi_k+\ell_{\xi_k}}{N}\in{\mathcal W}(\widehat{\mu})$ and ${\mathcal W}(\widehat{\mu})\cap\Z=\emptyset$.
\end{proof}

\bigskip
\subsection{Conclusion of the proof.} Denote by $$\widetilde{{\mathcal O}}=\left\{\xi\in[0, 1]: \frac1n\sum_{s=0}^{n-1}\left|M_{B_{s}}(\xi)\right|^2\geq \frac{1}{2m^{2}n}\right\}.$$
This is a compact set. 

\bigskip

\begin{Lem}\label{lem-positive2}
There exists $\epsilon_0>0$ such that for any $x\in\widetilde{{\mathcal O}}$, there exists an integer $k_{x}\in\Z$ such that 
$$|\widehat{\mu}(x+k_x)|^2\geq \epsilon_0>0.$$
\end{Lem}
\begin{proof}
For any $x\in \widetilde{{\mathcal O}}\subset{\mathcal O}$, from Lemma \ref{W-empty}, we can find $k_{x}\in\Z$ and $\epsilon_x>0$ such that
$$
|\widehat{\mu}(x+k_{x})|^2\ge \epsilon_x.
$$
As $\widehat{\mu}$ is continuous on ${\mathbb R}$, we can find  $\delta_x$ such that for all $|y|\le \delta_x$, we have
$$
|\widehat{\mu}(x+y+k_{x})|^2\ge \frac{\epsilon_x}{2}.
$$
As $\widetilde{{\mathcal O}} \subset \bigcup_{x\in \widetilde{{\mathcal O}}} B(x,\delta_x/2)$, by the compactness of $\widetilde{{\mathcal O}}$, we can find $x_1,...,x_N\in \widetilde{{\mathcal O}}$ such that $\widetilde{{\mathcal O}} \subset B(x_1,\delta_{x_1}/2)\cup...\cup B(x_N,\delta_{x_N}/2)$. We now take
$$
\delta_0 = \min\left\{\frac{\delta_{x_j}}{2}: j=1,...,N\right\},  \
\epsilon_0 = \min\left\{\frac{\epsilon_{x_j}}{2}: j=1,...,N\right\}.
$$
Now, $\delta_0$ and $\epsilon_0$ are positive and independent of $x\in \widetilde{{\mathcal O}}$. We claim that the stated property holds. Indeed, for any $x\in \widetilde{{\mathcal O}}$, $x\in B(x_j,\delta_{x_j}/2)$ for some $j=1,...,N$. 
Hence,
$$
|\widehat{\mu}(x+k_{x_j})|^2 = |\widehat{\mu}(x_j+(x-x_j)+k_{x_j})|^2\ge \frac{\epsilon_{x_j}}{2}\ge \epsilon_0.
$$
Therefore, we just redefine $k_{x} = k_{x_j}$ to obtain our desired conclusion. 
\end{proof}

\medskip

We are now ready to prove Theorem \ref{thm_main1}.

\medskip
\noindent{\it Proof of Theorem \ref{thm_main1}:} According to Theorem \ref{th-main}, it suffices to establish the average equi-positivity condition. i.e. there exists a constant  $c>0$ and $\delta>0$ such that for all $\xi\in [0, 1]$, there exists $k_{\xi}\in\Z$ such that 
\begin{equation}\label{eq-main-1}
|\widehat{\mu}(\xi+y+k_{\xi})|^2\geq \frac{c}{n}\sum_{s=0}^{n-1}|M_{{\mathcal B_s}}(\xi+y)|^2
\end{equation}
whenever $|y|<\delta$. 

\medskip

From Lemma \ref{lem-positive}, there exists $\delta>0$ such that for any $\xi\in[0, 1]$, there is an integer $\ell_{\xi}\in{\mathcal L}$ with the property that  $\frac{\xi+\ell_{\xi}}{N}\in \widetilde{{\mathcal O}}$ and 
\begin{equation}\label{eq-th1}
\left|M_{\frac{\D}{N}}(\xi+y+\ell_{\xi})\right|^2\geq \frac{1}{2mn^2}\sum_{s=0}^{n-1}|M_{\B_s}(\xi+y)|^2.  
\end{equation}
whenever $|y|<\delta$. We now take $\epsilon_0$ as in Lemma \ref{lem-positive2}, for  $x=\frac{\xi+\ell_{\xi}}{N}\in\widetilde{{\mathcal O}}$, there exists an integer $t_{\xi}\in\Z$ such that 
$$\left|\widehat{\mu}\left(\frac{\xi+\ell_{\xi}}{N}+t_{\xi}\right)\right|^2\geq\epsilon_0.$$
By the uniform continuity of $\widehat{\mu}$ over $\R$, we can find $\delta'>0$ such that
\begin{equation}\label{eq-th2}
\left|\widehat{\mu}\left(\frac{\xi+y+\ell_{\xi}}{N}+t_{\xi}\right)\right|^2\geq\frac{\epsilon_0}{2}
\end{equation}
whenever $|y|<\delta'$. Therefore, if $|y|$ is smaller than both $\delta$ and $\delta'$, (\ref{eq-th1}) and (\ref{eq-th2}) both holds. Hence, 
\begin{eqnarray*}
\left|\widehat{\mu}(\xi+y+\ell_{\xi}+Nt_{\xi})\right|^2&=&\left|M_{\frac{\D}{N}}(\xi+y+\ell_{\xi}+Nt_{\xi})\right|^2\left|\widehat{\mu}\left(\frac{\xi+y+\ell_{\xi}+Nt_{\xi}}{N}\right)\right|^2\\
&=&\left|M_{\frac{\D}{N}}(\xi+y+\ell_{\xi})\right|^2\left|\widehat{\mu}\left(\frac{\xi+y+\ell_{\xi}}{N}+t_{\xi}\right)\right|^2\\
&\geq&\frac{\epsilon_0}{4mn^2}\sum_{s=0}^{n-1}|M_{B_s}(\xi+y)|^2.
\end{eqnarray*}
Therefore, we just define $k_{\xi}=\ell_{\xi}+Nt_{\xi}\in\Z$ and $c=\frac{\epsilon_0}{4mn^2}$ to obtain our desired conclusion. \qquad$\Box$

\medskip

\section{$k$-stage Product-form}

This section will be devoted to understanding the higher stage product-forms we introduced in Definition \ref{def-prod-form-Had} and prove Theorem \ref{theorem_main2}. The idea of the proof will be to reduce the higher stage product-form back to the one-stage case. To begin, we can insert $\{0\}$ in missing levels and rewrite (1.2) as 
\begin{equation*}\label{eq_product-form_Hada_direct}
\left\{\begin{array}{lll}
\widetilde{\D}^{(0)}&=&\E_0\\
\ 
\widetilde{\D}^{(1)}& = &\bigcup_{d\in\widetilde{\D}^{(0)}} \left(d+ N\{0\}\right)\\
&\vdots&\\
\widetilde{\D}^{(\ell_1)}& = &\bigcup_{d_{\ell_1-1}\in\widetilde{\D}^{(\ell_1-1)}} \left(d_{\ell_1-1}+ N^{\ell_1}\widetilde{\E}_{\ell_1}(d_{\ell_1-1})\right)\\
&\vdots&\\
\widetilde{\D}^{(\ell)} &=& \bigcup_{d_{\ell-1}\in\widetilde{\D}^{(\ell-1)}} \left(d_{\ell-1}+ N^{\ell} \widetilde{{\mathcal E}}_{\ell}(d_{\ell-1})\right),
\end{array}\right.
\end{equation*}
where $\widetilde{\E}_{j}(d_{j-1})=\E_i(d_{j-1})$ if $j=\ell_1+...+\ell_i$, otherwise $\widetilde{\E}_{j}=\{0\}$. Then $\D=\widetilde{\D}^{(\ell)}$.
Denote 
$$\widetilde{{\mathcal L}}_j=\left\{\begin{array}{lll}
{\mathcal L}_j\quad \text{ if $j=\ell_1+...+\ell_i$;}\\
\ 
\{0\}\quad \text{otherwise.}
\end{array}\right.$$
As $(N, \{0\}, \{0\})$ is trivially a Hadamard triple,   $(N,\D, \widetilde{{\mathcal L}}_0\oplus....\oplus \widetilde{{\mathcal L}}_{\ell})$ is a $\ell$-stage  product-form Hadamard triple.
By doing so, we can without loss of generality assume that $\ell_1=\ell_2=....=\ell_k = 1$ in (\ref{eq_product-form_Hada-Section1}) of Definition (1.2) with the convention that some $\E_j(d)$ may be $\{0\}$.

\medskip


\medskip

Our proof starts with a general proposition. 

\begin{Prop}\label{proposition_hada-k}
Suppose that a digit set $\A = \A^{(k)}$ is generated by the following process:
\begin{equation}\label{A}
\left\{\begin{array}{lll}
 \A^{(0)}&=&{\mathcal C}_0\\
\A^{(1)}& = &\bigcup_{a_0\in\A^{(0)}} \left(a_0+ N {\mathcal C}_1(a_0)\right)\\
&\vdots&\\
\A^{(k)} &=& \bigcup_{a_{k-1}\in{\mathcal A}^{(k-1)}} \left(a_{k-1}+ N^{k} {\mathcal C}_{k}(a_{k-1})\right),
\end{array}\right.
\end{equation}
where $(N, {\mathcal C}_0, \widetilde{{\mathcal L}}_1)$ and $(N, {\mathcal C}_j(a_{j-1}),\widetilde{{\mathcal L}}_j)$ are Hadamard triples. Then $(N^k, \A, {\bf L})$ forms a Hadamard triple where 
${\bf L} =\bigoplus_{m=0}^{k}N^{k-m-1}\widetilde{{\mathcal L}}_m \subset\Z.$
\end{Prop}

\begin{proof}
We will prove by induction on $n$ that  $\bigoplus_{m=0}^{n}\frac{{\widetilde{{\mathcal L}}}_m}{N^{m+1}}$ is a spectrum of $\delta_{\A^{(n)}}$ for each $0\le n\le k$. When $n=k$, this is equivalent to our desired conclusion. 

 When $n=0$, as $\A^{(0)}={\mathcal C}_0$  and the direct sum is just $\frac{\widetilde{{\mathcal L}}_0}{N}$, the claim is true by the fact that $(N,{\mathcal C}_0,\widetilde{{\mathcal L}}_0)$ is a Hadamard triple. Suppose we have proved the statement for $n-1$. For the case $n$, we have $\A^{(n)} = \bigcup_{a_{n-1}\in\A^{(n-1)}} (a_{n-1}+N^n {\mathcal C}_n(a_{n-1}))$ and hence 
\begin{eqnarray}\label{A-spectrum-Section5}
&&\sum_{\lambda\in \oplus_{m=0}^{n}\frac{{\widetilde{{\mathcal L}}}_m}{N^{m+1}}}\left|M_{\A^{(n)}}\left(\xi+\lambda\right)\right|^2\nonumber\\
&=&\sum_{\lambda\in \oplus_{m=0}^{n}\frac{{{\mathcal L}}_m}{N^{m+1}}}\left|\frac{1}{\#\A^{(n-1)}}\sum_{a\in \A^{(n-1)}}e\left(-a(\xi+\lambda)\right)M_{{N^{n}{\mathcal C}_{n}(a)}}(\xi+\lambda)\right|^2\nonumber\\
&=&\sum_{\lambda_2\in \oplus_{m=0}^{n-1}\frac{{\widetilde{{\mathcal L}}}_m}{N^{m+1}}}\sum_{\lambda_1\in\frac{\widetilde{{\mathcal L}}_{n}}{N^{n+1}} }\left|\frac{1}{\#\A^{(n-1)}}\sum_{a\in \A^{(n-1)}}e\left(-a(\xi+\lambda_1+\lambda_2)\right)M_{{N^{n}{\mathcal C}_{n}(a)}}(\xi+\lambda_1)\right|^2.
\end{eqnarray}
Since $\lambda_2\subset \frac{1}{N^{n-1}}\Z$, $\lambda_2$ disappeared in $M_{{N^{n}{\mathcal C}_{n}(a)}}$ by the periodicity.  Note that from the induction hypothesis,  $\bigoplus_{m=0}^{n-1}\frac{\widetilde{{\mathcal L}}_m}{N^{m+1}}$ is a spectrum of $\delta_{\A^{(n-1)}}$, then  $\A^{(n-1)}$ is a spectrum of $\delta_{\oplus_{m=0}^{n-1}\frac{\widetilde{{\mathcal L}}_m}{N^{m+1}}}$ too.
\begin{eqnarray*}
&&\sum_{\lambda_2\in \oplus_{m=0}^{n-1}\frac{\widetilde{{\mathcal L}}_m}{N^{m+1}}}\left|\frac{1}{\#\A^{(n-1)}}\sum_{a\in \A^{(n-1)}}e\left(-a(\xi+\lambda_1+\lambda_2)\right)M_{{N^{n}{\mathcal C}_{n}(a)}}(\xi+\lambda_1)\right|^2\\
&=&\sum_{\lambda_2\in \oplus_{m=0}^{n-1}\frac{\widetilde{{\mathcal L}}_m}{N^{m+1}}}\frac{1}{(\#\A^{(n-1)})^2}\sum_{a_1, a_2\in \A^{(n-1)}}e\left(-(a_1-a_2)(\xi+\lambda_1+\lambda_2)\right)\\
&&\cdot M_{{N^{n}{\mathcal C}_{n}(a_1)}}(\xi+\lambda_1)\overline{M_{{N^{n}{\mathcal C}_{n}(a_2)}}(\xi+\lambda_1)}\\
&=&\frac{1}{(\#\A^{(n-1)})^2}\sum_{a_1, a_2\in \A^{(n-1)}}e\left(-(a_1-a_2)(\xi+\lambda_1)\right) M_{{N^{n}{\mathcal C}_{n}(a_1)}}(\xi+\lambda_1)\overline{M_{{N^{n}{\mathcal C}_{n}(a_2)}}(\xi+\lambda_1)}\\
&&\cdot \sum_{\lambda_2\in \oplus_{m=0}^{n-1}\frac{\widetilde{{\mathcal L}}_m}{N^{m+1}}}e\left(-(a_1-a_2)\lambda_2\right)\\
&=&\frac{1}{\#\A^{(n-1)}}\sum_{a\in \A^{(n-1)}}\left|M_{{N^{n}{\mathcal C}_{n}(a)}}(\xi+\lambda_1)\right|^2  \ \ \left(\mbox{since $\A^{(n-1)}$ is a spectrum of $\delta_{\oplus_{m=0}^{n-1}\frac{\widetilde{{\mathcal L}}_m}{N^{m+1}}}$}\right).
\end{eqnarray*}
Substituting it into \eqref{A-spectrum-Section5}, we have 
\begin{eqnarray*}
&&\sum_{\lambda\in \oplus_{m=0}^{n}\frac{\widetilde{{\mathcal L}}_m}{N^{m+1}}}\left|M_{\A^{(n)}}\left(\xi+\lambda\right)\right|^2\nonumber\\
&=&\frac{1}{\#\A^{(n-1)}}\sum_{a\in \A^{(n-1)}}\sum_{\lambda_1\in\frac{\widetilde{{\mathcal L}}_{n}}{N^{n+1}} }\left|M_{{N^{n}{\mathcal C}_{n}(a)}}(\xi+\lambda_1)\right|^2\\
&=&\frac{1}{\#\A^{(n-1)}}\sum_{a\in \A^{(n-1)}}1=1.
\end{eqnarray*}
This completes the proof
\end{proof}

Suppose now we are given a $k$-stage product-form as in Definition \ref{eq_product-form_Hada-Section1} with $\ell_1=...=\ell_k=1$. Next we will show that 
$$\textbf{D}=\D+N\D+...+N^{k-1}\D=\bigcup_{a\in\A}(a+N^{k}{\mathcal B}_a)$$
is a one-stage product-form Hadamard triple generated by Hadamard triples $(N^k,\A,\textbf{L}_1)$ and $(N^k, \B_a, \textbf{L}_2)$ for some $\A,\B_a,{\mathcal L}_1$ and ${\mathcal L}_2$. Recall first that $\D^{(n)}$ is the $n$-th stage digit set defined in the k-stage product-form, i.e. 
$$\D^{(n)}=\bigcup_{{d}_{n-1}\in {\D}^{(n-1)}}\Big(d_{n-1}+N^{n}\E_{n}(d_{n-1})\Big).
$$
We extend the definition of $\E_{n}(d)=\{0\}$ if $n<0$ and $n>k$, and 
\begin{equation}
\D^{(n)}=\{0\} \text{ when } n<0\quad \text{and}\quad {\D^{(n)}=\D^{(k)}=\D}  \text{ when } n\geq k,
\end{equation}
so that now $\D^{(n)}$ is defined for all $n\in\Z$.
We now define
\begin{equation}\label{A-m-Section5}
\textbf{D}^{(m)}=\D^{(m)}+ N\D^{(m-1)}+...+ N^{k-1}\D^{(m-k+1)},\quad 0\le m\le 2k-1,
\end{equation}
This implies that  $\textbf{D}^{(0)}=\D^{(0)}=\E_0$ and
\begin{eqnarray*}
\textbf{D}^{(2k-1)}&=&\D^{(2k-1)}+N\D^{(2k-2)}+...+N^{k-1}\D^{(k)}\\
&=&\D+N\D+...+N^{k-1}\D\\
&=&\textbf{D}.
\end{eqnarray*}
Note that
\begin{eqnarray}\label{eq-D-Section5}
\textbf{D}^{(m)}&=&\sum_{j=0}^{k-1}N^{j}\D^{(m-j)}\nonumber\\
& = &\sum_{j=0}^{k-1}N^{j}\left(\bigcup_{d_{m-j-1}\in \D^{(m-j-1)}}\left(d_{m-j-1}+N^{m-j}\E_{m-j}(d_{m-j-1})\right)\right)\nonumber\\
&=&\bigcup_{\sum_{j=0}^{k-1} N^{j}d_{m-j-1}\in \textbf{D}^{(m-1)}}\left(\sum_{j=0}^{k-1} N^{j}d_{m-j-1}+N^{m}\left(\sum_{j=0}^{k-1}\E_{m-j}(d_{m-j-1})\right)\right)\nonumber\\
&:=&\bigcup_{\textbf{d}_{m-1}\in \textbf{D}^{(m-1)}}\Big(\textbf{d}_{m-1}+N^{m}{\mathcal G}_{m}(\textbf{d}_{m-1})\Big)
\end{eqnarray}
where $\textbf{d}_{m-1}=\sum_{j=0}^{k-1} N^{j}d_{m-j-1}, {\mathcal G}_{m}(\textbf{d}_{m-1})=\sum_{j=0}^{k-1}\E_{m-j}({d}_{m-j-1})$. 
For any $\textbf{d}\in \textbf{D}^{(k-1)}$, we let 
$${\mathcal B}_{\textbf{d}}=\{b: \textbf{d}+N^k b\in \textbf{D}\}.$$
Then 
\begin{equation}\label{eqbfD}
    \textbf{D}=\bigcup_{\textbf{d}\in \textbf{D}^{(k-1)}}\left(\textbf{d}+N^k{\mathcal B}_{\textbf{d}}\right).
\end{equation}
\bigskip

\begin{theorem}\label{th_k-to-1-section5}
With reference to the above notation, there exists ${\bf L}_1$ and ${\bf L_2}$, both are subsets of integers, such that $(N^k, \textbf{D}^{(k-1)},{\bf L}_1)$, $(N^k, {\mathcal B}_{\textbf{d}},{\bf L}_2)$ and $(N^k, \textbf{D}^{(k-1)}+{\mathcal B}_{\textbf{d}},{\bf L}_1\oplus{\bf L}_2)$ are Hadamard triples for all $\textbf{d}\in \textbf{D}^{(k-1)}$. 
\end{theorem}

\begin{proof} The proof will be divided into three steps. 

\medskip

\noindent {\bf Step 1: Finding ${\bf L}_1$ for $(N^k, \textbf{D}^{(k-1)},{\bf L}_1)$ to form a Hadamard triple.}
From the equation \eqref{eq-D-Section5}, $\textbf{D}^{(k-1)}$ is generated in the following process:
\begin{equation}\label{A-section7}
\left\{\begin{array}{lll}
\textbf{D}^{(0)}&=&\E_0\\
\textbf{D}^{(1)}& = &\bigcup_{\textbf{d}_0\in\textbf{D}^{(0)}} \left(\textbf{d}_0+ N {\mathcal G}_1(\textbf{d}_0)\right)\\
&\vdots&\\
\textbf{D}^{(k-1)} &=& \bigcup_{\textbf{d}_{k-2}\in\textbf{D}^{(k-2)}} \left(\textbf{d}_{k-2}+ N^{k-1} {\mathcal G}_{k-1}(\textbf{d}_{k-2})\right),
\end{array}\right.
\end{equation}
where $\textbf{d}_{m-1}=\sum_{j=0}^{k-1} N^{j}d_{m-j-1}$ with $d_{m-j-1}\in\D^{(m-j-1)}$ and 
\begin{equation*}\label{eq-AG}
{\mathcal G}_{m}(\textbf{d}_{m-1})=\bigoplus_{j=0}^{k-1}\E_{m-j}({d}_{m-j-1})=\E_{m}(d_{m-1})\oplus...\oplus\E_{0}, \quad \forall\ 0\le m\le k-1 .
\end{equation*}
It follows from $\E_n(d_{n-1})=\{0\}$ for $n<0$. By  Definition \ref{def-prod-form-Had} (ii), we have 
$(N, {\mathcal G}_{m}(\textbf{d}_{m-1}), {\mathcal L}_0\oplus...\oplus{\mathcal L}_m)$ forms a Hadarmard triple. For any $0\le m\le k-1$, we let $\widetilde{{\mathcal L}_m}={\mathcal L}_0\oplus...\oplus{\mathcal L}_m\subset\Z.$ By Proposition \ref{proposition_hada-k}, defining $$\textbf{L}_1=\bigoplus_{m=0}^{k-1}N^{k-m-1}\widetilde{{\mathcal L}}_m\subset\Z.$$
 ensures us that $(N^{{k}}, \textbf{D}^{(k-1)}, \textbf{L}_1)$ is a Hadamard triple. 
 \bigskip
 
 \noindent {\bf Step 2: Finding ${\bf L}_2$ for $(N^k, \B_a,{\bf L}_2)$ to form a Hadamard triple.}
For any $0\le m\le k-1$, we define 
    $${\mathcal B}_{{\textbf{d}}}^{(m)}=\left\{b: {\textbf{d}}+N^kb\in \textbf{D}^{(k+m)}\right\}.$$
Since $\textbf{D}=\textbf{D}^{(2k-1)}$, we have 
$$\B_{\textbf{d}}=\{b: \textbf{d}+N^k b\in\textbf{D}^{2k-1}\}=\B_{\textbf{d}}^{(k-1)}.$$ From equation \eqref{eq-D-Section5}, we have the following process:
\begin{equation}\label{B-section7}
\left\{\begin{array}{lll}
\B_{\textbf{d}}^{(0)}&=&{\mathcal G}_k({\textbf{d}})\\
\ \quad\\
\B_{\textbf{d}}^{(1)}& = &\bigcup_{b_0\in\B_{\textbf{d}}^{(0)}} \left(b_0+ N {\mathcal G}_{k+1}(\textbf{d}+N^k b_0)\right)\\
&\vdots&\\
\B_{\textbf{d}}^{(k-1)} &=& \bigcup_{b_{k-2}\in\B_{\textbf{d}}^{(k-2)}} \left(b_{k-2}+ N^{k-1} {\mathcal G}_{2k-1}(\textbf{d}+N^kb_{k-2})\right),
\end{array}\right.
\end{equation}
and $\textbf{d}+N^k b_{m}\in \textbf{D}^{(k+m)}$ since $b_{m}\in \B_{\textbf{d}}^{(m)}$. Recall 
\begin{equation*}\label{A-m}
\textbf{D}^{(k+m)}=\D^{(k+m)}+ N\D^{(k+m-1)}+...+ N^{k-1}\D^{(m+1)}.
\end{equation*}
So we can write  $\textbf{d}+N^k b_{m}=\sum_{j=0}^{k-1}N^jd_{k+m-j}$ with $d_{k+m-j}\in \D^{(k+m-j)}$. As $\E_{n}=\{0\}$ for each $n>k$, so 
$${\mathcal G}_{k+m}(\textbf{d}+N^k b_{m})=\bigoplus_{j=0}^{k-1}\E_{k+m-j}({d}_{k+m-j-1})=\E_{k}(d_{k-1})\oplus...\oplus\E_{m+1}(d_m),\quad  \forall\ 0\le m\le k-1.$$
By  Definition \ref{def-prod-form-Had} (ii), we have 
$(N, {\mathcal G}_{k+m}(\textbf{d}+N^k b_{m}), {\mathcal L}_k\oplus...\oplus{\mathcal L}_{m+1})$ forms a Hadarmard triple. For any $m\geq0$, we let $\widetilde{{\mathcal L}_m}={\mathcal L}_k\oplus...\oplus{\mathcal L}_{m+1}\subset\Z.$ By Proposition \ref{proposition_hada-k}, defining $$\textbf{L}_2=\bigoplus_{m=0}^{k-1}N^{k-m-1}\widetilde{{\mathcal L}}_m\subset\Z.$$
 ensures us that $(N^{{k}}, \B_{\textbf{d}}, \textbf{L}_2)$ is a Hadamard triple. 
 
 \bigskip

 \noindent {\bf Step 3: $(N^k, \textbf{D}^{(k-1)}+\B_{\textbf{d}},{\bf L}_1+{\bf L}_2)$ forms a Hadamard triple.}  Combining \eqref{A-section7} and \eqref{B-section7}, $\textbf{D}^{(k-1)}+\B_{\textbf{d}}=\textbf{D}^{(k-1)}+\B_{\textbf{d}}^{(k-1)}$ is generated as following:
\begin{equation*}
\left\{\begin{array}{lll}
\textbf{D}^{(0)}+\B_{\textbf{d}}^{(0)}&=&{\mathcal E}_0+{\mathcal G}_{k}(\textbf{d})\\
\ \quad\\
\textbf{D}^{(1)}+\B_{\textbf{d}}^{(1)}& = &\bigcup_{\textbf{d}_0+b_0\in\textbf{D}^{(0)}+\B_{\textbf{d}}^{(0)}} \left(\textbf{d}_0+b_0+ N \left({\mathcal G}_1(\textbf{d}_0)+{\mathcal G}_{k+1}(\textbf{d}+N^kb_0)\right)\right)\\
&\vdots&\\
\textbf{D}^{(k-1)}+\B_{\textbf{d}}^{(k-1)}& = &\bigcup_{\textbf{d}_{k-2}+b_{k-2}\in\textbf{D}^{(k-1)}+\B_{\textbf{d}}^{(k-2)}} \left(\textbf{d}_{k-2}+b_{k-2}+ N^{k-1} \left({\mathcal G}_{k-1}(\textbf{d}_{k-1})+{\mathcal G}_{2k-1}(\textbf{d}+N^kb_{k-1})\right)\right),\\
\end{array}\right.
\end{equation*}
where 
$${\mathcal G}_m(\textbf{d}_{m})+{\mathcal G}_{k+m}(\textbf{d}+N^kb_{m})=\bigoplus_{j=0}^{m}\E_{j}(d_{j-1})\oplus\bigoplus_{j=m+1}^{k}\E_{j}(d_{j-1})$$
for some $d_j\in\D^{(j)}$. 
Hence $(N, {\mathcal G}_m(\textbf{d}_{m})+{\mathcal G}_{k+m}(\textbf{d}+N^kb_{m}), {\mathcal L}_0\oplus...\oplus{\mathcal L}_k)$ is a Hadarmard triple and so  is
$(N^k, \textbf{D}^{(k-1)}+\B_{\textbf{d}}, \textbf{L}_1\oplus\textbf{L}_2)$ by Proposition \ref{proposition_hada-k}.
\end{proof}

\medskip

The proof of Theorem \ref{theorem_main2} is now immediate. 

\medskip

\noindent{\it Proof of Theorem \ref{theorem_main2}.} Recall that the self-similar measure $\mu_{N,\D} = \mu_{N^k,{\bf D}}$. As we have written ${\bf D}$ as in (\ref{eqbfD}).
By Theorem \ref{th_k-to-1-section5}, we know that ${\bf D}$ is now generated by a one-stage product-form Hadamard triple as in Definition \ref{definition_product-form}. $\mu_{N^k,{\bf D}}$ is therefore a spectral measure by Theorem \ref{theorem_main1}.\qquad$\Box$

\section{Modulo Product-form and self-similar tiles}\label{section-mod}

In \cite{LLR2013} and \cite{LLR2017}, the authors introduced the modulo product-forms and showed that they cover all tile digit sets of $N = p^{\alpha}q$. The product-form was defined via modulo action along the cyclotomic polynomial factors.  We would like to show that modulo product-form is also the type of product-forms formulated in Definition \ref{definition_product-form}. 

\medskip

Let us first recall how modulo product-form are defined. The cyclotomic polynomial $\Phi_d$ is the minimal polynomial of $e^{2\pi i/d}$ in $\Z[x]$.  Given a set $\E$ such that it is decomposed into a direct sum of $k$ sets $\E = \E_0\oplus....\oplus\E_{k}$ and ${\mathfrak T}$ a set of cyclotomic polynomials $\Phi_d$ such that $\Phi_d(x)$ divides $P_{\E}(x)$. Let ${\mathfrak S}_j = \{d: \Phi_d\in{\mathfrak T}, \  d>1,  \ \Phi_d (x) \ |  \ P_{\E_j}(x)\}$. We know that $\bigcup_{j=0}^k{\mathfrak S}_j = {\mathfrak T}.$ Define
$$
m_j = \mbox{l.c.m.} \{d: \Phi_d\in{\mathfrak T}, \  d>1,  \ \Phi_d (x) \ |  \ P_{\E_i}(x) \ \mbox{for some} \ i = 0,...,j\}
$$
For a given positive integers $\ell_1,\ell_2,...,\ell_k$ and $N\ge 2$,  we define for $j = 0,...,k$
\begin{equation}\label{kernel_polynomial}
 K^{(j)}(x)  = \prod_{i=0}^j\prod_{d\in{\mathfrak S}_i} \Phi_d (x^{N^{\ell_1+...+\ell_i}})
   \end{equation}
\begin{equation}\label{kernel_nj}
n_j = \mbox{l.c.m.} \{d:  \  d>1,  \ \Phi_d(x) \ |  K^{(j)}(x)\}.
\end{equation}
Then the {\it modulo product-form} of $(\E, {\mathfrak T})$  is the  $\D = \D^{(k)}$ generated inductively by 
\begin{equation}\label{eq-D}
\left\{\begin{array}{lll}
\D^{(0)}&=&\E_0\\
\ 
\D^{(1)}& = &\D^{(0)} + N^{\ell_1}\E_1 \ (\mbox{mod}  \ n_1)\\
&\vdots&\\
\D^{(k)}& = &\D^{(k-1)} + N^{\ell_1+...+\ell_k}\E_{k} \ (\mbox{mod}  \ n_k)
\end{array}\right.
\end{equation}

\begin{Lem}\label{lemn-j}
With reference to above notation, $n_j = m_j N^{\ell_1+...+\ell_j}$ and $K^{(k)}(x) \ | \  P_{\D}(x)$. 
\end{Lem}

\begin{proof}
A property of cyclotomic polynomials is that 
\begin{equation}\label{prop_cyclo}
\Phi_d (x^s) = \left\{\begin{array}{ll}  \Phi_{ds}(x)&  \mbox{if $s$ is a factor of $d$}\\ \Phi_d(x)\Phi_{ds}(x) &  \mbox{if $s$ is a not factor of $d$} \end{array}\right.
\end{equation}
From this property, we know that for all $d$ such that $\Phi_d\in{\mathfrak T}$ and $\Phi_d(x) \ | \ P_{\E_i}(x)$ with $0\le i\le j$,  $\Phi_{dN^{\ell_1+..+\ell_j}}(x)$ is a cyclotomic factor in $K^{(j)}(x)$. Hence, $n_j\ge m_jN^{\ell_1+...+\ell_j}$. Conversely, if $\Phi_e$ is a cyclotomic polynomial dividing $K^{(j)}(x)$,  (\ref{prop_cyclo}) implies that $e = d m$ where $d\in{\mathfrak S}_i$ and $m$ divides $N^{\ell_1+...+\ell_i}$ for some $0\le i\le j$. Hence, $e$ divides $m_jN^{\ell_1+...+\ell_j}$. This shows that $n_j = m_jN^{\ell_1+...+\ell_j}$. 

\medskip

For the second statement, we note that $\D^{(j)} = \D^{(j-1)}+ N^{\ell_1+...+\ell_j} \E_j$ (mod $n_j$) is equivalent to saying 
\begin{equation}\label{equation_D_j}
P_{\D^{(j)}}(x) = P_{\D^{(j-1)}}(x)P_{\E_{j}} (x^{\ell_1+...+\ell_j}) + (x^{n_j}-1)Q(x). 
\end{equation}
We prove by induction that $K^{(j)}(x) \ | \ P_{\D^{(j)}}(x)$ and our statement will follow when $j = k$. It is clearly true when $j = 0$. Suppose that $K^{(j-1)}(x) \ | \ P_{\D^{(j-1)}}(x)$. We note that $\Phi_d(x^{\ell_1+...+\ell_j}) \ | \ P_{\E_j}(x^{\ell_1+...+\ell_j})$. Hence, by induction hypothesis, $K^{(j)}(x)$ divides the first term on the right hand side of (\ref{equation_D_j}). As $x^{n_j}-1 = \prod_{d|n_j}\Phi_d(x)$, $K^{(j)}(x)$ divides $x^{n_j}-1$. Hence, $K^{(j)}(x)$ divides $P_{\D^{(j)}}(x)$ follows. 
\end{proof}

\begin{Prop}\label{th-mpf}
A {\it modulo product-form} of $(\E, {\mathfrak T})$  is a $k-$stage product-form digit with $\D = \D^{(k)}$  generated in the following process:
\begin{equation*}\label{eq_product-form_Hada-Section1}
\left\{\begin{array}{lll}
\ \D^{(0)}&=&\E_0\\
\ 
\D^{(1)}& = &\bigcup_{d_0\in\D^{(0)}} \Big(d_0+ N^{\ell_1} {\mathcal E}_1(d_0)\Big)\\
\
\D^{(2)}& = &\bigcup_{d_0\in\D^{(1)}} \Big(d_1+ N^{\ell_1+\ell_2} {\mathcal E}_2(d_1)\Big)\\
&\vdots&\\
\ \D^{(k)} &=& \bigcup_{d_{k-1}\in{\mathcal D}^{(k-1)}} \Big(d_{k-1}+ N^{\ell_1+...+\ell_{k}} {\mathcal E}_k(d_{k-1})\Big),
\end{array}\right.
\end{equation*}
where $\E_j(d)\equiv \E_j\pmod{m_j}$ for each $d\in \D^{(j-1)}$.
\end{Prop}
\begin{proof}
By Lemma \ref{lemn-j},
$$\D^{(j)}\equiv\D^{(j-1)}+N^{\ell_1+...+\ell_j}\E_j\pmod{n_j},\quad n_j=m_jN^{\ell_1+...+\ell_j}.$$
So we can write $\D^{(j)}$ as 
$$\D^{(j)}=\left\{d+N^{\ell_1+...+\ell_j}(e+m_j\cdot z(d,e)): d\in \D^{(j-1)}, e\in\E_j \right\}=\bigcup_{d\in \D^{(j-1)}}\left(d+N^{\ell_1+...+\ell_j}\E_j(d)\right)$$
where $z(d, e)\in\Z$ and $\E_j(d)=\{e+m_j\cdot z(d,e): e\in \E_j\}$. From the definition of $m_j$, we have 
$$\E_j(d)\equiv \E_j\pmod{m_j}.$$
\end{proof}

\begin{Def}
We say that $\D$ is a {\it first-order modulo product-form of $N>1$} if $\E \equiv \Z_N$ and ${\mathfrak T} = \{\Phi_d: d|N, d>1\}$. 

\medskip

Inductively, We say that $\D$ is a {\it $m$-th order modulo product-form of $N>1$} if $\E$ is a $(m-1)$-th order product-form and ${\mathfrak T} = \{\Phi_d: \Phi_d(x) \ |  \ K_{m-1}^{(k)}(x)\}$ where $K_{(m-1)}^{(k)}(x)$ is the polynomial in (\ref{kernel_polynomial}) for the $(m-1)$-th order product-form.
\end{Def}

In the above definition, it says that we first generate a first order product-form $\E_1$. It may be possible to rearrange $\E_1$ into another direct summands. From there, we can use the same procedure to generate a second order product-form, and the process can go on. 
We now let $\E_p = \{0,1,...,p-1\}$. The following theorem was obtained in \cite{LLR2017}.

\begin{theorem} [Theorem 5.5, \cite{LLR2017}] Let $N = p^2q$ where $p,q$ are prime numbers and let $\D$ be an integer subset of $\#\D = N$ with g.c.d $(\D) = 1$. Then $\D$ is a tile digit set of $N$ (i.e. the attractor $K (N,\D)$ is a self-similar tile) if and only if $\D$ is a modulo product-form of one of the following three types:
\begin{enumerate}
    \item $\E = \E_p\oplus (p\E_q)\oplus (pq\E_p)$;
    \item $\E = \E_p\oplus p^{2\ell}\E_q\oplus (p^{2\ell-1}\E_p)$, for some $\ell\geq 1$;
    \item $\E = \E_q\oplus (q\E_p) \oplus (pq\E_p)$,
\end{enumerate}
and ${\mathfrak T}$ are the set of all cyclotomic polynomials dividing $P_{\E}(x)$ in each case.
\end{theorem}

\medskip

More generally, the $p^{\alpha}q$-tile digit sets are also completely classified. 

\begin{theorem} [Theorem 5.8, \cite{LLR2017}] \label{theorem5.8LLR} Let $N = p^{\alpha}q$ where $p,q$ are prime numbers and let $\D$ be an integer subset of $\#\D = N$ with g.c.d $(\D) = 1$. Then $\D$ is a tile digit set of $N$ (i.e. the attractor $K (N,\D)$ is a self-similar tile) if and only if $\D$ is a modulo product-form of one of the following three types:
\begin{enumerate}
    \item $\E = \E_p\oplus (p\E_q)\oplus  \bigoplus_{j=1}^{\alpha-1}(p^{j}q\E_p)$,
    
    \medskip
    
    \item $\E = \E_p\oplus (p^{\alpha(M+1)+k}\E_q)\oplus\bigoplus_{j=1}^{\alpha-1} p^{\alpha M_j+j}\E_p$, where $M_j\geq 0$, $\forall j\ge 2$, $M =\max\{M_j : 1\le j\le \alpha-1\}$ and $k = \max\{j : M_j = M\}$.

\medskip

    \item $\E = \E_q\oplus \bigoplus_{j=0}^{\alpha-1} (p^{j}q\E_p)$
\end{enumerate}
and ${\mathfrak T}$ are the set of all cyclotomic polynomials dividing $P_{\E}(x)$ in each case.
\end{theorem}

Indeed, the case (i) and (iii) in both $p^2q$ and $p^{\alpha}q$, $\E = \{0,1,...,N-1\}$, so the modulo product-form are easily seen to be from a strict product-form. While for case (ii), they were understood as higher order product-forms  in \cite{LLR2017}. However, the following proposition provides a more natural perspective of case (ii) as first-order modulo product-form or an $\alpha$-stage product-form from a direct sum of $\Z_{p^{\alpha}q}$. 

\begin{Prop}\label{high-to-one}
Suppose $\D$ is a modulo product-form of $\E$ in the case (ii) of Theorem \ref{theorem5.8LLR}. Then $q^{M}\E$ is a direct product-form of $p^{\alpha}q$ and $q^{M}\D$ is the first order modulo product-form of 
$$
(q^{M}\E_p)\oplus p^{\alpha+k}\E_q\oplus \bigoplus_{j=1}^{\alpha-1} (q^{M-M_j}p^{j}\E_p)\equiv\Z_{p^{\alpha}q}.
$$
\end{Prop}

\begin{proof}
We multiple $q^{M}$ to $\E$ in (ii) of Theorem \ref{theorem5.8LLR}. We have 
$$
q^{M}\E = (q^{M}\E_p)\oplus N^{M}(p^{\alpha+k}\E_q)\oplus\bigoplus_{j=1}^{\alpha-1} N^{M_j} \left(q^{M-M_j}p^{j}\E_p\right)
$$
As $p,q$ are relatively prime, 
\begin{equation}\label{eqmod}
q^{M}\E_p\equiv \E_p  \ (\mbox{mod}  \ p), \  \ p^{\alpha+k}\E_q\equiv  p^{\alpha}\E_q \  (\mbox{mod} \ q) \ \ \mbox{and} \ \  q^{M-M_j}p^{j}\E_p\equiv p^{j}\E_p  \ (\mbox{mod} \  p^{j+1}).
\end{equation}  Therefore, $(q^{M}\E_p)\oplus p^{\alpha+k}\E_q\oplus \bigoplus_{j=1}^{\alpha-1} (q^{M-M_j}p^{j}\E_p)\equiv \Z_{p^{\alpha}q}$. This shows that $q^{M}\E$ is a first order direct product-form.

\medskip

Multiplying $q^{M+1}$ to the modulo product-form generating process of $\D$, $q^{M+1}\D$ is generated as
$$
\left\{\begin{array}{lll}
\D^{(0)}&=&q^{M}\E_p\\
\ 
\D^{(1)}& = &\D^{(0)} + N^{\ell_1+M} (p^{\alpha+k}\E_q) \ (\mbox{mod}  \ n_1)\\
\D^{(2)}&=&\D^{(1)} + N^{\ell_1+\ell_2+ M_2} (q^{M-M_2}p\E_q) \ (\mbox{mod}  \ n_2)\\
&\vdots&\\
\D^{(\alpha)}& = &\D^{(\alpha-1)} + N^{\ell_1+...+\ell_k+M_{\alpha} }(q^{M-M_{\alpha}}p^{\alpha-1}\E_q) \ (\mbox{mod}  \ n_{\alpha})
\end{array}\right.
$$
where $n_j$ are defined in the same way in (\ref{kernel_nj}). This shows that $q^{M+1}\D$ is a first order modulo product-form given in the statement. We have completed the proof. 
\end{proof}

The proposition shows that higher order product-form is not necessarily required to describe all tile digit sets of $p^{\alpha}q$ as long as we do not require the g.c.d. of $\D$ to be equal to 1. However, it is hard to predict if there could be a reduction for  higher order product-forms if $N$ contains many primes.  We are now ready to prove Theorem \ref{theorem_FHL}.

\medskip

\noindent{\it Proof of Theorem \ref{theorem_FHL}.} Let $\D$ be a tile digit set of $N= p^{\alpha}q$. Our goal is to show that $\mu_{N,\D}$ is a spectral measure. This is sufficient to show that $\D$ forms a $\alpha$-stage Hadamard triple. Using Proposition \ref{th-mpf} and Proposition \ref{high-to-one}, we know  that all cases (i)-(iii) are $\alpha$-stage product-form. We now construct ${\mathcal L}_0,{\mathcal L}_1...,{\mathcal L}_{\alpha}$ so that it forms a product-form Hadamard triple.

\medskip

\noindent (i) In case (i) of Theorem \ref{theorem5.8LLR}, we let $\widetilde{\E}_0 = \E_p$, $\widetilde{\E}_1 = p\E_q$ and $\widetilde{\E}_{j} = p^{j-1}q\E_p$. Then $(N, \widetilde{\E}_0, p^{\alpha-1}q\E_p)$, $(N, \widetilde{\E}_1, p^{\alpha-1}\E_q)$ and $(N, \widetilde{\E}_j, p^{\alpha-j}\E_p)$ are Hadamard triples. Define  
$$
{\mathcal L}_0 = p^{\alpha-1}q\E_p, \ {\mathcal L}_1 =p^{\alpha-1}\E_q, \ ..., \  {\mathcal L}_j = p^{\alpha-j}\E_p.
$$
It is routine to check that $(N, \widetilde{\E}_0\oplus...\oplus \widetilde{\E}_{m}, {\mathcal L}_0\oplus...\oplus{\mathcal L}_{m})$ and $(N, \widetilde{\E}_{m}\oplus...\oplus \widetilde{\E}_{\alpha}, {\mathcal L}_m\oplus...\oplus{\mathcal L}_{\alpha})$ are Hadamard triples as well. Therefore, it forms a Hadamard triple. 

\medskip

\noindent (ii) In case (ii) of Theorem \ref{theorem5.8LLR}, we will show that $q^{M}\D$ is generated by a product-form Hadamard triple via 
$$
(q^{M}\E_p)\oplus N^{M}(p^{\alpha+k}\E_q)\oplus\bigoplus_{j=1}^{\alpha-1} N^{M_j} \left(q^{M-(M_j)}p^{j}\E_p\right) : = \widetilde{\E}_0\oplus\widetilde{\E}_1\oplus...\oplus \widetilde{\E}_{\alpha}.
$$
 Define  
$$
{\mathcal L}_0 = p^{\alpha-1}q\E_p, \ {\mathcal L}_1 =\E_q, \ ..., \  {\mathcal L}_j = p^{\alpha-j}q\E_p.
$$
 (\ref{eqmod}) meaning that we have the same kind of Hadamard triple structure. Hence,  $(N, \widetilde{\E}_j,{\mathcal L}_j)$ are Hadamard triples for all $j = 0,1,...,\alpha$, so are its products. 

\medskip
\noindent (iii) Case (iii) of Theorem \ref{theorem5.8LLR} is based on the same construction of (i), so we omit the detail. \qquad$\Box$


\medskip
\section{Tiling, four-digit self-similar measures and some examples}\label{examples-section}

In this section, we will demonstrate  how one can generate product-form Hadamard triples. We will first prove Theorem \ref{theorem_four_digit} which provides many basic examples of product-form digit sets. Then we demonstrate one can generate product-form Hadamard triples by tilings. Finally,  we will provide several examples that are new under current settings. 

\medskip

\subsection{Four-digit self-similar measures.} We now offer a proof for our four-digit cases.

\medskip

\noindent{\it Proof of Theorem \ref{theorem_four_digit}.} Let $N = 2^{\beta}m $ where $\beta\ge 1$ and $m$ is odd. Recall also that the digit set $\D = \{0,a,2^t\ell,a+2^t\ell'\}$ where $a,\ell,\ell'$ are odd numbers and $t$ is not divisible by $\beta$. Hence, we can write $t = \beta k +r$ where $k\ge 0$ and $r \in\{1,...,\beta-1\}$. Now, we can write 
$$
\D = (\{0\}+ 2^{\beta k} \{0,2^r\ell\})\cup (\{a\}+ 2^{\beta k} \{0,2^r\ell'\})
$$
Multiplying $m^k$ on both sides, we obtain
$$
m^k\D = (\{0\}+ N^k \{0,2^r\ell\})\cup (\{a m^k\}+ N^k\{0,2^r\ell'\}).
$$
Note that $\mu_{N,\D}$ is spectral if and only if  $ \mu_{N,m^k\D}$ is spectral. We now claim that $m^k\D$ is a one-stage product-form  of $N$ generated by some Hadamard triples. 

\medskip

Let $\A = \{0,am^k\}$, $\B_0 = \{0,2^r\ell\}$ and $\B_a = \{0,
2^r\ell'\}$. Define also ${\mathcal L}_1 = \{0, \frac{N}{2}\}$ and ${\mathcal L}_{2} = \{0, N2^{r-1}\}$ We can check immediately that $(N,\A,{\mathcal L}_1)$ and $(N,\B_i,{\mathcal L}_2)$ are Hadamard triples for all $i = 0$ or $a$.   Moreover, $(N, \A\oplus \B_i, {\mathcal L}_1\oplus{\mathcal L}_2)$ also forms a Hadamard triple (note that this is false if $r = 0$). 
The conclusion now follows from Theorem \ref{theorem_main1}.
\qquad$\Box$

\subsection{Coven-Meyerowitz tiling theory.} Let us recall the theory by Coven-Meyerowitz and {\L}aba concerning tiling on cyclic groups. Let $\Phi_d(x)$ be the $d^{\rm th}$ cyclotomic polynomial, which is the minimal polynomial of $e^{2\pi i /d}$ in $\Z[x]$.   For a finite set of integers $\A$,  $P_{\A}(z) = \sum_{a\in\A}z^a$.  



\medskip
 We define 
${\mathcal S}_{\A}$ to be the set of all prime-powers $s=p^{\alpha}$ such that  $s|N$  and  $\Phi_{s}(x)| P_{\A}(x)$. Coven and Meyerowitz introduced the following two conditions for $P_{\A}$ which guarantees tilings for $\A$ for some $\Z_N$.  

\medskip
$(T1)$: $P_{\A}(1) = \prod_{s\in {\mathcal S}_{\A}} \Phi_{s} (1)$.

\medskip

$(T2)$: \mbox{For all distinct prime powers $s_1,...,s_k\in{\mathcal S}_A$, we have  $\Phi_{s_1...s_k}(x)| P_{\A}(x)$}

Coven-Meyerowitz and {\L}aba proved the following theorem.

\begin{theorem} Let $\A$ (mod $N$) be a subset of $\Z_N$. Then 
\medskip

1. [Coven-Meyerowitz] Suppose that $\A$ satisfies $(T1)$ and $(T2)$. Then $\A$ tiles $\Z_N$. Conversely, if $\A$ is tile the cyclic group $\Z_N$, then $(T1)$ holds. Moreover, if $\#\A$ contains at most two prime factors, then $(T2)$ holds. 

\medskip

2. [{\L}aba] Suppose that $\A$ satisfies $(T1)$ and $(T2)$. Then $\A$  a spectral set in $\Z_N$. $\A$ admits a spectrum 
$$
{\mathcal L}_{\A} = \left\{\sum_{s \in\S_{\A}} {\epsilon_s}\frac{N}{s}: \epsilon_s\in \{0,1,...,p-1\}  \ \mbox{if} \  s=p^{\alpha}  \right\}.
$$
\end{theorem}

\medskip

A major open problem is to determine if all tiles in a cyclic group $\Z_N$ satisfy $(T1)$ and $(T2)$. If that is the case, we can also settle the ``tiling implies spectrality" direction of the Fuglede's conjecture on $\R^1$. We say that a cyclic group $\Z_N$ is {\it CM-regular} if for all $\A\subset \Z_N$ such that it tiles $\Z_N$, $\A$ satisfies condition $(T1)$ and $(T2)$. For all finite groups $\Z_N$ for which Fuglede's conjecture has been verified so far, $\Z_N$ are all CM-regular. Now, the following corollary is immediate.

\begin{corollary}\label{corollary_CM}
    Let $\Z_N$ be a CM-regular cyclic group and suppose that 
    $$
    \Z_N \equiv \A_1\oplus\A_2\oplus...\oplus\A_k \ (\mbox{mod} \ N)
    $$
    Then $(N,\D,{\mathcal L}_{\A_1}\oplus...\oplus {\mathcal L}_{\A_k})$ forms a $k$-stage product-form Hadamard triple if we define $\D$ as in (\ref{def-prod-form-Had}) and  $\mu_{N,\D}$ is a spectral measure. 
\end{corollary}

\begin{remark}
\begin{enumerate}
    \item {\rm We remark that if we are not working on $CM$-regular cyclic groups, but we assume that  $\A$, $\B$ and $\A\oplus \B$  satisfies $(T1)$ and $(T2)$ with $\A\oplus \B$ tiles $\Z_N$ for some $N\in\N$, then we can still obtain the conclusion in Corollary \ref{corollary_CM}}.
    \item {\rm In the special case that $\A\oplus \B = \Z_N$, the self-similar measure becomes the Lebesgue measure supported on the self-similar tiles, so we can conclude that they are spectral sets.}   
\end{enumerate}
\end{remark}

\medskip

\begin{example}\label{example24}
Let $N = 24 = 2^3\cdot 3$. The digit set 
$$
\D = \{0,1,16,17\} = \{0,1, 2^4, 1+2^4\}
$$
By Theorem \ref{theorem_four_digit}, $\mu_{N,\D}$ is a spectral measure and 
$$
3\D = \{0,3\}\oplus 24 \{0,2\}
$$ 
is a product-form digit set with the corresponding $\D_0 = \{0,3\}\oplus \{0,2\}$ which has g.c.d equals 1. This means that, by Theorem \ref{thm_main1}, we can find a spectrum of the form $ \frac{1}{24}\{0,1\}+ \Lambda$, where $\Lambda$ is an integer subset, for $\mu_{24,3\D}$. Dividing by 3, $\mu_{24,\D}$ has a spectrum $ \frac{1}{72}\{0,1\}+ \frac13\Lambda$. 

\medskip

Note that $\D$ is not a tile of $\Z_{24}$.  To see this, we can factorize $P_{\D}$
$$
P_{\D}(x) = \Phi_2(x)\Phi_2(x^{16}) = \Phi_2(x)\Phi_{2^5}(x).
$$
As $2^5$ does not divide $24$, ${\D}$ cannot tile $\Z_{24}$ (it still tiles other cyclic groups). As Fuglede's conjecture for $\Z_{24}$ holds \cite{MK2017},  $\D$ is not a spectral set for $\Z_{24}$. Therefore, there is no ${\mathcal L}$ such that $(24,\D,{\mathcal L})$ forms a Hadamard triple in the ordinary sense. 
This example shows that even though   $\D$ is distinct residue modulo 24 and we cannot form an ordinary compatible pair using any integer set ${\mathcal L}$,  but it is still a spectral measure by our new product-form Hadamard triple consideration. 
\end{example}

\medskip

\begin{example}
This example appeared first in \cite{CM1999} and later also in \cite{FHL2015}. Let $N = 72 = 2^3\cdot 3^2$ and 
$$
\A = \{0,8,16,18,26,34\}, \ \B  = \{0,5,6,9,12,29,33,36,42,48,53,57\}.
$$
Then $\A\oplus \B\equiv \Z_{72}$. It was shown in \cite{FHL2015} that $\A\oplus\B$ cannot be reduced to strict product  (i.e $\A\oplus \B= \{0,1,...,71\}$.  Therefore, it was not known if the self-similar tile generated by the product-form $\D_r = \A\oplus (72)^r\B$ or its modulo product-form is a spectral set from \cite{FHL2015}. However, with our Theorem \ref{theorem_main1}, $\mu_{72,\D_r}$ are all spectral measures and the self-similar tiles are all spectral sets. 
\end{example}

\medskip

\section{Conclusion}

The main conclusion of this paper is to introduce product-form Hadamard triples generated by two different Hadamard triples of integer scale $N$. We showed that the self-similar measure formed by these digits are spectral measures.   Furthermore, we have also made two important observations reducing the complicated cases back to the one stage case. 
\begin{enumerate}
    \item Higher stage product-forms are also reduced to the one-stage product-form by considering some higher power of $N$. 
    \item Even if the digit sets do not look like a product-form, after multiplying some factors, the new digit set also form a product-form Hadamard triple (see Example \ref{example24}, Proof of Theorem \ref{theorem_FHL}).  
\end{enumerate}
In view of this, we would like to end this paper by stating the following conjecture which seems to be a reasonable description of all digit sets of $N$ that the self-similar measure $\mu_{N,\D}$ is a spectral measure. This conjecture will also provide the missing third condition in the {\L}aba-Wang conjecture. 

\begin{conjecture}\label{modified_LW}
Let $N\ge 2$ be a positive integer. Suppose that the self-similar measure $\mu_{N,\D}$ is a spectral measure. Then there exists $m>0$ and $r\ge 0$ such that $m\D$ forms a $r$-stage product-form Hadamard triple with respect to $N$ for some ${\mathcal L}_1\oplus...\oplus{\mathcal L}_r$.
\end{conjecture}

If $r=0$, it means that the Hadamard triple is an ordinary Hadamard triple in Definition \ref{def_hada}. By Theorem \ref{th_k-to-1-section5}, the conclusion in the conjecture can also be  reformulated as  there exists $k$ such that 
$${\bf D}_k = (m\D)+N(m\D)+...+ N^{k-1}(m\D)$$
forms a one-stage product-form Hadamard triple with respect to $N^k$. In our next paper, we will show that the four-digit case does confirm this conjecture.

\end{document}